\documentclass[a4paper,11pt,reqno]{amsart}
\usepackage{amsfonts,amsmath,amssymb}
\usepackage{graphicx}
\usepackage{fixmath}
\usepackage{hyperref}
\usepackage{tikz}
\usetikzlibrary{positioning}
\usepackage{cite}
\usepackage{a4wide}
\usepackage{cancel}
\usepackage{mathtools}

\newtheorem{theo}{Theorem}
\newtheorem{corollary}{Corollary}

\newtheorem{lemma}{Lemma}

\newtheorem{example}{Example}

\newcommand{\Prob}{\mathbb{P}}
\newcommand{\Ex}{\mathbb{E}}

\DeclareMathOperator{\Aut}{\mathrm{Aut}}

\begin{document}
\thispagestyle{plain}
\title{Distinct Fringe Subtrees in Random Trees}

\author{Louisa Seelbach Benkner and Stephan Wagner}

\email{seelbach@eti.uni-siegen.de, stephan.wagner@math.uu.se, swagner@sun.ac.za}

\thanks{A short version of this paper appeared in the Proceedings of LATIN 2020 \cite{SeelbachWagner20}. \newline
This project has received funding from the  European Unions Horizon 2020 research and innovation programme under the Marie Sk\l odowska-Curie grant agreement No 731143, the DFG research project
LO 748/10-2 (QUANT-KOMP) and the Knut and Alice Wallenberg Foundation.}

\begin{abstract}
A fringe subtree of a rooted tree is a subtree induced by one of the vertices and all its descendants. We consider the problem of estimating the number of distinct fringe subtrees in two types of random trees: simply generated trees and families of increasing trees (recursive trees, $d$-ary increasing trees and generalized plane-oriented recursive trees). We prove that the order of magnitude of the number of distinct fringe subtrees (under rather mild assumptions on what `distinct' means) in random trees with $n$ vertices is $n/\sqrt{\log n}$ for simply generated trees and $n/\log n$ for increasing trees.
\end{abstract}
\maketitle   
  
\section{Introduction}

A subtree of a rooted tree that consists of a vertex and all its descendants is called a \emph{fringe subtree}. Fringe subtrees are a natural object of study in the context of random trees, and there are numerous results for various random tree models, see for example \cite{aldous91, dennertgr10, devroye14, FengM10}.

Fringe subtrees are of particular interest in computer science: One of the most important and widely used lossless compression methods for rooted trees is to represent a tree as a directed acyclic graph, which is obtained by merging vertices that are roots of identical fringe subtrees. This compressed representation of the tree is often shortly referred to as \emph{minimal DAG} and its size (number of vertices) is the number of distinct fringe subtrees occurring in the tree. Compression by minimal DAGs has found numerous applications in various areas of computer science, as for example in compiler construction \cite[Chapter~6.1 and 8.5]{AhoSU86}, unification \cite{PatersonW78}, symbolic model checking (binary decision diagrams) \cite{Bry92}, information theory \cite{GanardiHLS19, ZhangYK14} and XML compression and querying \cite{BuGrKo03,FrGrKo03}.

In this work, we investigate the number of fringe subtrees in random rooted trees. So far, this problem has mainly been studied with respect to the number of \emph{distinct} fringe subtrees, where two fringe subtrees are considered as distinct if they are distinct as members of the particular family of trees. 
In \cite{FlajoletSS90}, Flajolet, Sipala and Steyaert proved that, under very general assumptions, the expected number of distinct fringe subtrees in a tree of size $n$ drawn \emph{uniformly at random} from some given family of trees is asymptotically equal to $c \cdot n/\sqrt{\log n}$, where the constant $c$ depends on the particular family of trees. In particular, their result covers uniformly random plane trees (where the constant $c$ evaluates to $c=\sqrt{(\log 4)/\pi}$) and uniformly random binary trees (with $c=2\sqrt{(\log 4)/\pi}$). The result of Flajolet et al.~was extended to  uniformly random $\Sigma$-labelled unranked trees in \cite{MLMN13} (where $\Sigma$-labelled means that each vertex of a tree is assigned a label from a finite alphabet $\Sigma$ and unranked means that the label of a vertex does not depend on its degree or vice versa) and reproved with a different proof technique in \cite{RalaivaosaonaW15} in the context of simply generated families of trees.

Another probabilistic tree model with respect to which the number of distinct fringe subtrees has been studied is the \emph{binary search tree model}: a random binary search tree of size $n$ is a binary search tree built by inserting the keys $\{1, \ldots, n\}$ according to a uniformly chosen random permutation on $\{1, \ldots, n\}$. Random binary search trees are of particular interest in computer science, as they naturally arise for example in the analysis of the Quicksort algorithm, see \cite{Drmota09}. In \cite{FlajoletGM97}, Flajolet, Gourdon and Martinez proved that the expected number of distinct fringe subtrees in a random binary search tree of size $n$ is $O(n/\log n)$. This result was improved in \cite{Devroye98} by Devroye, who showed that the asymptotics $\Theta(n/\log n)$ holds. 
In a recent paper by Bodini, Genitrini, Gittenberger, Larcher and Naima \cite{BodiniGGLN20}, the result of Flajolet, Gourdon and Martinez was reproved, and it was shown that the average number of distinct fringe subtrees in a random recursive tree of size $n$ is $O(n/\log n)$ as well. 
Moreover, the result of Devroye was generalized from random binary search trees to a broader class of random ordered binary trees in \cite{SeelbachLo18}, where the problem of estimating the expected number of distinct fringe subtrees in random binary trees was considered in the context of leaf-centric binary tree sources, which were introduced in \cite{KiefferYS09, ZhangYK14}  as a general framework for modelling probability distributions on the set of  binary trees of size $n$.

In this work, we consider two types of random trees: Random simply generated trees (as a general concept to model uniform probability distributions on various families of trees) and specific families of increasing trees (recursive trees, $d$-ary increasing trees and generalized plane oriented recursive trees), which in particular incorporate the binary search tree model (for the precise definitions see Sections~\ref{sec:simply-generated} and \ref{sec:increasing}).

Specifically, we investigate the number of ``distinct'' fringe subtrees with respect to these random tree models under a generalized interpretation of ``distinctness'', which allows for many different interpretations of what ``distinct'' trees are. To give a concrete example of different notions of distinctness, consider the family of $d$-ary trees where each vertex has $d$ possible positions to which children can be attached (for instance, if $d = 3$, a left, a middle and a right position). The following three possibilities lead to different interpretations of when two trees are regarded the same:
\begin{itemize}
\item the order and the positions of branches matter,
\item the order of branches matters, but not the positions to which they are attached,
\item neither the order nor the positions matter.
\end{itemize}

\begin{figure}[t]

\tikzset{level 1/.style={sibling distance=10mm}}
		\tikzset{level 2/.style={sibling distance=10mm}}
		\tikzset{level 3/.style={sibling distance=10mm}}  
		\tikzset{level 4/.style={sibling distance=10mm}}

		\centering
	\begin{tikzpicture}[scale=0.8,auto,swap,level distance=10mm,  ]
		\node[circle, inner sep = 2pt,fill=black] (eps) {} 
		child {node[circle, fill=black, inner sep = 2pt,minimum size = 1pt]{}
		child {node[circle, fill=black, inner sep = 2pt,minimum size = 1pt]{}}
	child {edge from parent [draw=none]node[circle, draw=none, inner sep = 2pt,minimum size = 1pt]{}}}
		child {node[circle, fill=black, inner sep = 2pt,minimum size = 1pt]{}}
;
\node[circle, inner sep = 2pt,fill=black, right=1.2cm  of eps] (eps2) {} 
		child {node[circle, fill=black, inner sep = 2pt,minimum size = 1pt]{}
		child {edge from parent [draw=none]node[circle, draw=none, inner sep = 2pt,minimum size = 1pt]{}}
	child {node[circle, fill=black, inner sep = 2pt,minimum size = 1pt]{}}}
		child {node[circle, fill=black, inner sep = 2pt,minimum size = 1pt]{}}
;
\node[circle, inner sep = 2pt,fill=black, right =1.2cm of eps2] (eps3) {} 
		child {node[circle, fill=black, inner sep = 2pt,minimum size = 1pt]{}
		}
		child {node[circle, fill=black, inner sep = 2pt,minimum size = 1pt]{}child {node[circle, fill=black, inner sep = 2pt,minimum size = 1pt]{}}
	child {edge from parent [draw=none]node[circle, draw=none, inner sep = 2pt,minimum size = 1pt]{}}}
;
\node[circle, inner sep = 2pt,fill=black, right=1.2cm of eps3] (eps4) {} 
		child {node[circle, fill=black, inner sep = 2pt,minimum size = 1pt]{}
		}
		child {node[circle, fill=black, inner sep = 2pt,minimum size = 1pt]{}child {edge from parent [draw=none]node[circle, draw=none, inner sep = 2pt,minimum size = 1pt]{}}
	child {node[circle, fill=black, inner sep = 2pt,minimum size = 1pt]{}}}
;

\node[circle, inner sep = 2pt,fill=black, right=5cm of eps4] (eps5) {} 
		child {node[circle, fill=black, inner sep = 2pt,minimum size = 1pt]{}
		child {node[circle, fill=black, inner sep = 2pt,minimum size = 1pt]{}}}
		child {node[circle, fill=black, inner sep = 2pt,minimum size = 1pt]{}}
;
\node[circle, inner sep = 2pt,fill=black, right=1.2cm of eps5] (eps6) {} 
		child {node[circle, fill=black, inner sep = 2pt,minimum size = 1pt]{}
		}
		child {node[circle, fill=black, inner sep = 2pt,minimum size = 1pt]{}
	child {node[circle, fill=black, inner sep = 2pt,minimum size = 1pt]{}}}
;
		\end{tikzpicture}
		\caption{Four distinct binary trees (left), and the two distinct plane trees associated to them (right), which are in turn identical as unordered trees}
		\label{fig:distinct-trees}
		\end{figure}
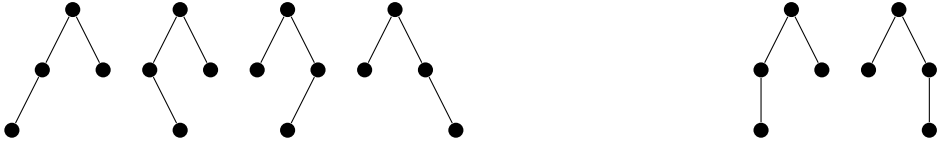

\begin{figure}[t]

\tikzset{level 1/.style={sibling distance=15mm}}
		\tikzset{level 2/.style={sibling distance=10mm}}
		\tikzset{level 3/.style={sibling distance=10mm}}  
		\tikzset{level 4/.style={sibling distance=10mm}}
		
		\flushleft
		\begin{minipage}{0.35\textwidth}
	\begin{tikzpicture}[scale=1,auto,swap,level distance=10mm,  ]
		\node[circle, inner sep = 2pt,fill=black] (eps) {} 
		child {node[circle, fill=black, inner sep = 2pt,minimum size = 1pt]{}
		child {node[circle, fill=black, inner sep = 2pt,minimum size = 1pt]{}		
		child {edge from parent[draw=none] node[circle, draw=none, inner sep = 2pt,minimum size = 1pt]{}}
		child {node[circle, fill=black, inner sep = 2pt,minimum size = 1pt]{}}}
	child {node[circle, fill=black, inner sep = 2pt,minimum size = 1pt]{}}}
		child {node[circle, fill=black, inner sep = 2pt,minimum size = 1pt]{}
		child {node[circle, fill=black, inner sep = 2pt,minimum size = 1pt]{}
		}
	child {node[circle, fill=black, inner sep = 2pt,minimum size = 1pt]{}
	child {node[circle, fill=black, inner sep = 2pt,minimum size = 1pt]{}}
	child {edge from parent[draw=none] node[circle, draw=none, inner sep = 2pt,minimum size = 1pt]{}}}	
		}
;
		\end{tikzpicture}
		\end{minipage}
		\begin{minipage}{0.5\textwidth}
		\begin{tikzpicture}[scale=0.65,auto,swap,level distance=10mm ]
		\node[circle, inner sep = 2pt,fill=black] (eps) {} 
		child {node[circle, fill=black, inner sep = 2pt,minimum size = 1pt]{}
		child {node[circle, fill=black, inner sep = 2pt,minimum size = 1pt]{}		
		child {edge from parent[draw=none] node[circle, draw=none, inner sep = 2pt,minimum size = 1pt]{}}
		child {node[circle, fill=black, inner sep = 2pt,minimum size = 1pt]{}}}
	child {node[circle, fill=black, inner sep = 2pt,minimum size = 1pt]{}}}
		child {node[circle, fill=black, inner sep = 2pt,minimum size = 1pt]{}
		child {node[circle, fill=black, inner sep = 2pt,minimum size = 1pt]{}
		}
	child {node[circle, fill=black, inner sep = 2pt,minimum size = 1pt]{}
	child {node[circle, fill=black, inner sep = 2pt,minimum size = 1pt]{}}
	child {edge from parent[draw=none] node[circle, draw=none, inner sep = 2pt,minimum size = 1pt]{}}}	
		}
;
\tikzset{level 1/.style={sibling distance=10mm}}
\node[circle, inner sep = 2pt,fill=black, right=2cm of eps] (eps2) {} 
		child {node[circle, fill=black, inner sep = 2pt,minimum size = 1pt]{}
		child {edge from parent[draw=none] node[circle, draw=none, inner sep = 2pt,minimum size = 1pt]{}	}
	child { node[circle, fill=black, inner sep = 2pt,minimum size = 1pt]{}}}
		child {node[circle, fill=black, inner sep = 2pt,minimum size = 1pt]{}
		}
;
		\node[circle, inner sep = 2pt,fill=black, right=1.2cm of eps2] (eps3) {} 
		child {node[circle, fill=black, inner sep = 2pt,minimum size = 1pt]{}
		}
		child {node[circle, fill=black, inner sep = 2pt,minimum size = 1pt]{}
		child { node[circle, fill=black, inner sep = 2pt,minimum size = 1pt]{}	}
	child { edge from parent[draw=none] node[circle,draw=none, inner sep = 2pt,minimum size = 1pt]{}}
		}
;
\node[circle, inner sep = 2pt,fill=black, right=1.2cm of eps3] (eps4) {} 
		child {node[circle, fill=black, inner sep = 2pt,minimum size = 1pt]{}
		}
	child { edge from parent[draw=none] node[circle, draw=none, inner sep = 2pt,minimum size = 1pt]{}}
;
\node[circle, inner sep = 2pt,fill=black, right=0.5cm of eps4] (eps5) {} 
		child {edge from parent[draw=none] node[circle, draw=none, inner sep = 2pt,minimum size = 1pt]{}
		}
	child { node[circle, fill=black, inner sep = 2pt,minimum size = 1pt]{}}
;
\node[circle, inner sep = 2pt,fill=black, right=1.2cm of eps5] (eps) {} 
;
\node[below=1.75cm of eps3](eps55){\qquad (i) Distinct binary fringe subtrees};
\node[draw=none, below=0.25cm of eps55](eps125){};
		\end{tikzpicture}
		\begin{tikzpicture}[scale=0.65,auto,swap,level distance=10mm ]
		\node[circle, inner sep = 2pt,fill=black] (eps) {} 
		child {node[circle, fill=black, inner sep = 2pt,minimum size = 1pt]{}
		child {node[circle, fill=black, inner sep = 2pt,minimum size = 1pt]{}		
		child {node[circle, fill=black, inner sep = 2pt,minimum size = 1pt]{}}}
	child {node[circle, fill=black, inner sep = 2pt,minimum size = 1pt]{}}}
		child {node[circle, fill=black, inner sep = 2pt,minimum size = 1pt]{}
		child {node[circle, fill=black, inner sep = 2pt,minimum size = 1pt]{}
		}
	child {node[circle, fill=black, inner sep = 2pt,minimum size = 1pt]{}
	child {node[circle, fill=black, inner sep = 2pt,minimum size = 1pt]{}}
	}	
		}
;
\tikzset{level 1/.style={sibling distance=10mm}}
\node[circle, inner sep = 2pt,fill=black, right=2cm of eps] (eps2) {} 
		child {node[circle, fill=black, inner sep = 2pt,minimum size = 1pt]{}
	child { node[circle, fill=black, inner sep = 2pt,minimum size = 1pt]{}}}
		child {node[circle, fill=black, inner sep = 2pt,minimum size = 1pt]{}
		}
;
		\node[circle, inner sep = 2pt,fill=black, right=1.2cm of eps2] (eps3) {} 
		child {node[circle, fill=black, inner sep = 2pt,minimum size = 1pt]{}
		}
		child {node[circle, fill=black, inner sep = 2pt,minimum size = 1pt]{}
		child { node[circle, fill=black, inner sep = 2pt,minimum size = 1pt]{}	}
		}
;
\node[circle, inner sep = 2pt,fill=black, right=1.2cm of eps3] (eps4) {} 
		child {node[circle, fill=black, inner sep = 2pt,minimum size = 1pt]{}
		}
;
\node[circle, inner sep = 2pt,fill=black, right=1.2cm of eps4] (eps) {} 
;
\node[below=1.75cm of eps3](eps55){\qquad (ii) Distinct plane fringe subtrees};
\node[draw=none, below=0.25cm of eps55](eps125){};
		\end{tikzpicture}
		\begin{tikzpicture}[scale=0.65,auto,swap,level distance=10mm ]
		\node[circle, inner sep = 2pt,fill=black] (eps2) {} 
		child {node[circle, fill=black, inner sep = 2pt,minimum size = 1pt]{}
		child {node[circle, fill=black, inner sep = 2pt,minimum size = 1pt]{}		
		child {node[circle, fill=black, inner sep = 2pt,minimum size = 1pt]{}}}
	child {node[circle, fill=black, inner sep = 2pt,minimum size = 1pt]{}}}
		child {node[circle, fill=black, inner sep = 2pt,minimum size = 1pt]{}
		child {node[circle, fill=black, inner sep = 2pt,minimum size = 1pt]{}
		}
	child {node[circle, fill=black, inner sep = 2pt,minimum size = 1pt]{}
	child {node[circle, fill=black, inner sep = 2pt,minimum size = 1pt]{}}
	}	
		}
;
\tikzset{level 1/.style={sibling distance=10mm}}
		\node[circle, inner sep = 2pt,fill=black, right=2cm of eps2] (eps3) {} 
		child {node[circle, fill=black, inner sep = 2pt,minimum size = 1pt]{}
		}
		child {node[circle, fill=black, inner sep = 2pt,minimum size = 1pt]{}
		child { node[circle, fill=black, inner sep = 2pt,minimum size = 1pt]{}	}
		}
;
\node[circle, inner sep = 2pt,fill=black, right=1.2cm of eps3] (eps4) {} 
		child {node[circle, fill=black, inner sep = 2pt,minimum size = 1pt]{}
		}
;
\node[circle, inner sep = 2pt,fill=black, right=1.2cm of eps4] (eps) {} 
;
\node[below right =1.75cm and -2.5cm of eps4](eps55){(iii) Distinct unordered fringe subtrees};
		\end{tikzpicture}
		\end{minipage}
		\caption{A binary tree (left) and (i) the six distinct binary trees, (ii) the five distinct plane trees and (iii) the four distinct unordered trees represented by its fringe subtrees }
		\label{fig:distinct-trees-2}
		\end{figure}
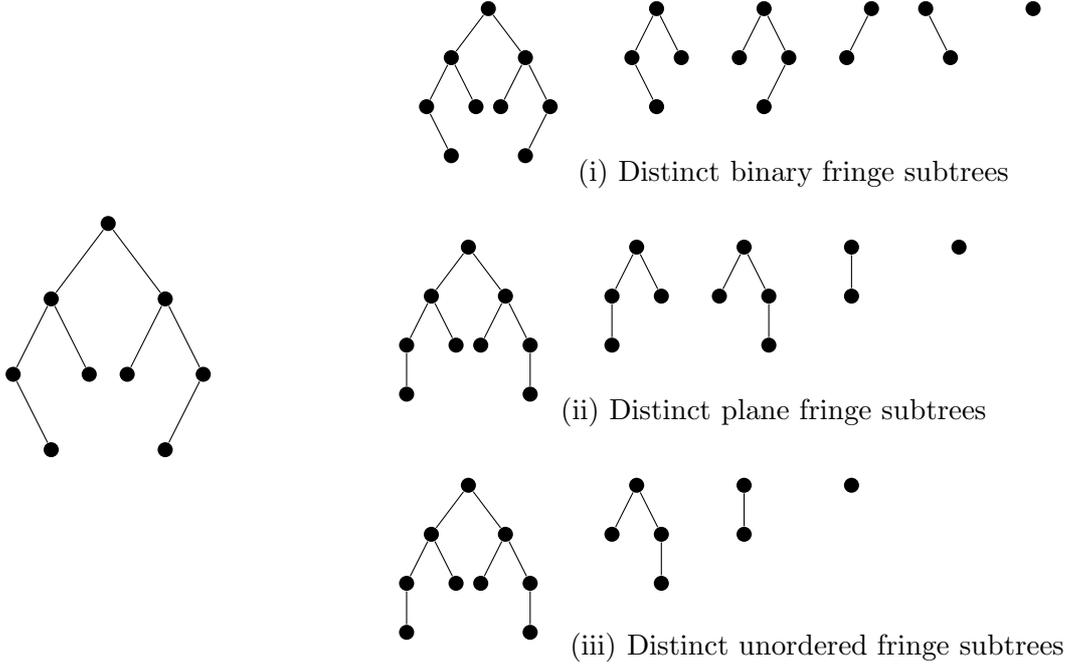

See Figure~\ref{fig:distinct-trees} and  Figure~\ref{fig:distinct-trees-2} for an illustration.
In order to cover all these cases, we only assume that the trees of order $k$ within the given family $\mathcal{F}$ of trees are partitioned into a set $\mathcal{I}_k$ of isomorphism classes for every $k$. The quantity of interest is the total number of isomorphism classes that occur among the fringe subtrees of a random tree with $n$ vertices. 
The following rather mild assumptions turn out to be sufficient for our purposes:

\begin{itemize}
\item[(C1)] We have $\limsup_{k \to \infty} \frac{\log |\mathcal{I}_k|}{k} = C_1 < \infty$.
\item[(C2)] There exist subsets $\mathcal{J}_k \subseteq \mathcal{I}_k$ of isomorphism classes and a positive constant $C_2$ such that
\begin{itemize}
\item[(C2a)] a random tree in the family $\mathcal{F}$ with $k$ vertices belongs to a class in $\mathcal{J}_k$ with probability $1 - o(1)$  as $k \to \infty$, and
\item[(C2b)] the probability that a random tree in $\mathcal{F}$ with $k$ vertices lies in a fixed isomorphism class $I \in \mathcal{J}_k$ is never greater than $e^{- C_2k + o(k)}$.

\end{itemize}
\end{itemize}
Note that (C2a) and (C2b) imply that $|\mathcal{I}_k| \geq |\mathcal{J}_k| \geq e^{C_2k - o(k)}$, thus we have $C_1 \geq C_2 > 0$.
Under the conditions (C1) and (C2), we prove the following general statement (for the definitions of offspring distributions and Galton--Watson processes, see Section~\ref{sec:simply-generated}):

\begin{theo}\label{thm:master-theorem-simply-generated}
Let $\mathcal{F}$ be a simply generated family of trees with a partition into isomorphism classes that satisfies (C1) and (C2), and let $\xi$ be the offspring distribution of the Galton--Watson process corresponding to $\mathcal{F}$, which satisfies $\Ex(\xi)=1$ and $\mathbb{V}(\xi)=\sigma^2<\infty$.
Let $A_n$ denote the total number of different isomorphism classes represented by the fringe subtrees of a random tree  $T_n$ of size $n$ drawn randomly from the specific family $\mathcal{F}$. Set $\kappa =\sqrt{2/(\pi\sigma^2)}$. We have
\begin{itemize}
\item[(i)] $\displaystyle \frac{\kappa \sqrt{C_2} n}{\sqrt{\log n}} (1+o(1)) \leq \Ex(A_n) \leq \frac{\kappa \sqrt{C_1} n}{\sqrt{\log n}} (1+o(1))$,
\item[(ii)] $\displaystyle \frac{\kappa \sqrt{C_2} n}{\sqrt{\log n}} (1+o(1)) \leq A_n \leq \frac{\kappa \sqrt{C_1} n}{\sqrt{\log n}} (1+o(1))$ with high probability (i.e., with probability tending to $1$ as $n \to \infty$).
\end{itemize}
\end{theo}

The same also applies to families of increasing trees, of which binary search trees and recursive trees are special cases: we obtain essentially the same statement, with the order of magnitude being $\frac{n}{\log n}$ rather than $\frac{n}{\sqrt{\log n}}$.

\begin{theo}\label{thm:master-theorem-increasing}
Let $\mathcal{F}$ be one of the ``very simple families'' of increasing trees (recursive trees, $d$-ary increasing trees, or gports, see Section~\ref{sec:increasing}) with a partition into isomorphism classes that satisfies (C1) and (C2). Let $A_n$ denote the total number of different isomorphism classes represented by the fringe subtrees of a random tree $T_n$ of size $n$ drawn from $\mathcal{F}$. Set $\kappa = \frac{1}{1+\alpha}$, where $\alpha=0$ in the case of recursive trees, $\alpha=1/r$ for some constant $r>0$ in the case of gports, and $\alpha=-1/d$ for $d$-ary increasing trees. We have
\begin{itemize}
\item[(i)] $\displaystyle \frac{\kappa C_2 n}{\log n} (1+o(1)) \leq \Ex(A_n) \leq \frac{\kappa C_1 n}{\log n} (1+o(1))$,
\item[(ii)] $\displaystyle \frac{\kappa C_2 n}{\log n} (1+o(1)) \leq A_n \leq \frac{\kappa C_1 n}{\log n} (1+o(1))$ with high probability.
\end{itemize}
\end{theo}

As our main application of these theorems, we investigate the number of distinct unordered trees represented by the fringe subtrees of a random tree. This question arises quite naturally for example in the context of XML compression: Here, one distinguishes between document-centric XML, for which the corresponding XML document trees are ordered, and data-centric XML, for which the corresponding XML document trees are unordered. Understanding the interplay between ordered and unordered structures has thus received considerable attention in the context of XML (see for example~\cite{AbiteboulBV15, BonevaCS15, ZhangDW15}). In particular, in \cite{LohreyMR17}, it was investigated whether tree compression can benefit from unorderedness. For this reason, \emph{unordered minimal DAGs} were considered. An unordered minimal DAG of a tree is a directed acyclic graph obtained by merging vertices that are roots of fringe subtrees which are identical as unordered trees. From such an unordered minimal DAG, an unordered representation of the original tree can be uniquely retrieved. The size of this compressed representation is the number of distinct unordered trees represented by the fringe subtrees occurring in the tree. So far, only some worst-case estimates comparing the size of a minimal DAG to the size of its corresponding unordered minimal DAG are known: among other things, it was shown in \cite{LohreyMR17} that the size of an unordered minimal DAG of a binary tree can be exponentially smaller than the size of the corresponding (ordered) minimal DAG. 

However, no average-case estimates comparing the size of the minimal DAG of a tree to the size of the corresponding unordered minimal DAG are known so far. In particular, in \cite{LohreyMR17} it is stated as an open problem to estimate the expected number of distinct unordered trees represented by the fringe subtrees of a uniformly random binary tree of size $n$ and conjectured that this number asymptotically grows as $\Theta(n/\sqrt{\log n})$.
 
In this work, as one of our main theorems, we settle this open conjecture by proving upper and lower bounds of order $n/\sqrt{\log n}$ for the number of distinct unordered trees represented by the fringe subtrees of a tree of size $n$ drawn randomly\ from a simply generated family of trees, which hold both in expectation and with high probability. For uniformly random binary trees, our result reads as follows:

\begin{theo}\label{thm:unordered-binary}
Let $K_n$ denote the number of distinct unordered trees represented by the fringe subtrees of a uniformly random binary tree of size $n$. Then for $c_1 \approx 1.0591261434$ and $c_2 \approx 1.0761505454$, we have 
\begin{itemize}
\item[(i)]$\displaystyle c_1 \frac{n}{\sqrt{\log n}}(1+o(1))\leq \Ex(K_n) \leq c_2 \frac{n}{\sqrt{\log n}}(1+o(1))$,
\item[(ii)]$\displaystyle c_1 \frac{n}{\sqrt{\log n}}(1+o(1))\leq K_n \leq c_2\frac{n}{\sqrt{\log n}}(1+o(1))$ with high probability.
\end{itemize}
\end{theo}

Our approach can also be used to obtain analogous results for  random recursive trees, $d$-ary increasing trees and generalized plane oriented recursive trees, though the order of magnitude changes to $\Theta(n/\log n)$. Again, we have upper and lower bounds in expectation and with high probability. For binary increasing trees, which are equivalent to binary search trees, our result reads as follows:

\begin{theo}\label{thm:unorderedbst}
Let $K_n$ be the total number of distinct unordered trees represented by the fringe subtrees of a random binary search tree of size $n$. For two constants $c_3 \approx 1.5470025923$ and $c_4 \approx 1.8191392203$, the following holds:
\begin{enumerate}
\item[(i)] $\displaystyle c_3 \frac{n}{\log n} (1+o(1)) \leq \Ex(K_n) \leq c_4 \frac{n}{\log n} (1+o(1))$,
\item[(ii)] $\displaystyle c_3 \frac{n}{\log n} (1+o(1)) \leq K_n \leq c_4 \frac{n}{\log n} (1+o(1))$ with high probability.
\end{enumerate}
\end{theo}
Both Theorem~\ref{thm:unordered-binary} and Theorem~\ref{thm:unorderedbst} were already given in the conference version \cite{SeelbachWagner20} of this paper\footnote{In the conference version \cite{SeelbachWagner20}, we consider \emph{full binary trees}, i.e., ordered trees such that each vertex has exactly two or zero descendants, whereas in this version, we allow binary trees to have (left- and right-) unary vertices. The respective probabilistic models are equivalent, see Section \ref{subsec:unordered}.}.
Additionally, we improve several existing results on the number of fringe subtrees in random trees. We show that the estimate from \cite[Theorem 4]{FlajoletSS90} and \cite[Theorem 3.1]{RalaivaosaonaW15} on the number of distinct fringe subtrees (as members of the particular family) in simply generated trees does not only hold in expectation, but also with high probability (see Theorem~\ref{thm:ordered-simply-generated}). Furthermore, we improve the lower bound on the number of distinct binary trees represented by the fringe subtrees of a random binary search tree: 
\begin{theo}\label{thm:orderedbst}
Let $H_n$ be the total number of distinct fringe subtrees in a random binary search tree of size $n$. For two constants $c_5 \approx 2.4071298335$ and $c_6 \approx 2.7725887222$, the following holds:
\begin{enumerate}
\item[(i)] $\displaystyle c_5 \frac{n}{\log n} (1+o(1)) \leq \Ex(H_n) \leq c_6 \frac{n}{\log n} (1+o(1))$,
\item[(ii)] $\displaystyle c_5 \frac{n}{\log n} (1+o(1)) \leq H_n \leq c_6 \frac{n}{\log n} (1+o(1))$ with high probability.
\end{enumerate}
\end{theo}
The upper bound in part (i) can already be found in \cite{FlajoletGM97} and \cite{Devroye98}. Moreover, a lower bound of the form $\Ex(H_n)\geq c n/\log (n)(1+o(1))$ was already shown in \cite{Devroye98} for the constant $c=(\log 3)/2 \approx 0.5493061443$ and in \cite{SeelbachLo18} for the constant $c \approx 0.6017824584$. So our new contributions in this special case are part (ii) and the improvement of the lower bound on $\Ex(H_n)$. Again, Theorem~\ref{thm:orderedbst} was already given in the conference version \cite{SeelbachWagner20} of this paper. 

Finally, we solve an open problem from \cite{BodiniGGLN20}, by proving that the number of distinct fringe subtrees in a random recursive tree of size $n$ is $\Theta(n/\log n)$ in expectation and with high probability (see Theorem~\ref{thm:unorderedrec}), thus showing a matching lower bound to the upper bound proved in \cite{BodiniGGLN20}.

\section{Preliminaries}

Let $t$ be a tree. We define the \emph{size} $|t|$ of $t$ as its number of vertices. Moreover, for a vertex $v$ of $t$, we denote with $\deg(v)$ the \emph{(out-)degree} of $v$, i.e.,~its number of children, and with $d_k(t)$ we denote the number of vertices of degree $k$ of $t$. 
A \emph{fringe subtree} of a tree $t$ is a subtree consisting of a vertex and all its descendants. For a tree $t$ and a given vertex $v$, let $t(v)$ denote the fringe subtree of $t$ rooted at $v$. For a family of trees $\mathcal{F}$, we will denote the subset of trees of size $k$ belonging to $\mathcal{F}$ by $\mathcal{F}_k$. Some important families of trees we will consider below are the following:

\begin{itemize}
\item \emph{Plane Trees:}
We write $\mathcal{T}$ for the family of plane trees, i.e.,~ordered rooted trees where each vertex has an arbitrary number of descendants, which are ordered from left to right. Moreover, we let $\mathcal{T}_k$ denote the set of plane trees of size $k$. 
\item \emph{Binary Trees:} The family of binary trees is the family of rooted ordered trees, such that each vertex has either (i) no children, (ii) a single left child, (iii) a single right child, or (iv) both a left and a right child. In other words, every vertex has two possible positions to which children can be attached. 
\item \emph{$d$-ary Trees:}
Binary trees naturally generalize to $d$-ary trees, for $d \geq 2$: a $d$-ary tree is an ordered tree where every vertex has $d$ possible positions to which children can be attached. Thus, the degree of a vertex $v$ of a $d$-ary tree is bounded above by $d$ and there are $\binom{d}{k}$ types of vertices of degree $k$ for $0 \leq k \leq d$: For example, if $d=3$, a vertex of degree $k=2$ can have a left and a middle child, a left and a right child, or a right and a middle child. Every $d$-ary tree can be considered as a plane tree by simply forgetting the positions to which the branches of the vertices are attached, respectively, by not distinguishing between different vertex types of degree $k$ for every $k \leq d$. This yields a partition of the set of $d$-ary trees into isomorphism classes, where two $d$-ary trees are considered as isomorphic if they correspond to the same plane tree. 
\item \emph{Unordered Trees:} An unordered tree is a rooted tree without an ordering on the descendants of the vertices. Every ordered tree can be considered as an unordered tree by simply forgetting the ordering on its vertices. This again yields a partition of the particular family of ordered trees into isomorphism classes, where two ordered trees are considered as isomorphic if they correspond to the same unordered tree. 
\item \emph{Labelled Trees:} A labelled tree of size $n$ is an unordered rooted tree whose vertices are labelled with the numbers $1,2,\ldots, n$. If we only take the shape of the tree into account, we can consider a labelled tree as an unordered tree: This yields a partition of the family of labelled trees into isomorphism classes, where we consider two labelled trees as isomorphic if their tree shapes are identical as unordered trees.
Furthermore, note that the labelling on the vertices of a labelled tree implicitly yields an ordering on the children of a vertex, if we sort them e.g.~in ascending order according to their labels. Thus, we can consider a labelled tree as a plane tree as well, if we first order the children of each vertex according to their labels, and then take only the shape of the tree into account.
\end{itemize}

\subsection{Simply generated families of trees and Galton--Watson trees}\label{sec:simply-generated}

A general concept to model various families of  trees is the concept of \emph{simply generated families of trees}:
It was introduced by Meir and Moon in \cite{MeirM78} (see also \cite{Drmota09, Janson12}). The main idea is to assign a weight to every plane tree $t \in \mathcal{T}$ which depends on the numbers $d_0(t), \ldots, d_{|t|}(t)
$ of vertices of degree $k$ for $0 \leq k \leq |t|$.
Let $(\phi_m)_{m \geq 0}$ denote a sequence of non-negative real numbers (called the \emph{weight sequence}), and let
\begin{align*}
\Phi(x)=\sum_{m \geq 0}\phi_mx^m.
\end{align*}
We define the \emph{weight} $w(t)$ of a plane tree $t$ as
\begin{align*}
w(t) =\prod_{v \in t}\phi_{\deg(v)}= \prod_{m\geq 0}\phi_m^{d_m(t)}.
\end{align*}
Moreover, let 
\begin{align*}
y_n=\sum_{t \in \mathcal{T}_n}w(t)
\end{align*}
denote the sum of all weights of plane trees of size $n$. It is well known that the generating function $Y(x) = \sum_{n \geq 1} y_n x^n$ satisfies
$$Y(x) = x \Phi(Y(x)).$$
A weight sequence $(\phi_m)_{m\geq 0}$ induces a probability mass function $P_{\Phi}: \mathcal{T}_n \to [0,1]$ on the set of plane trees of size $n$ by
\begin{align*}
P_{\Phi}(t) = \frac{w(t)}{y_n}
\end{align*}
for every $n\geq 0$ with $y_n>0$. We will tacitly assume that $y_n>0$ holds whenever we consider random plane trees of size $n$. A family $\mathcal{F}$ of trees is called simply generated if it can be described by a weight sequence $(\phi_k)_{k \geq 0}$, and random elements are generated according to the probability mass function $P_{\Phi}$. 

\begin{example}\label{ex:plane}
The family of plane trees is a simply generated family of trees with weight sequence $(\phi_k)_{k \geq 0}$ defined by $\phi_k=1$ for every $k \geq 0$: Thus, every plane tree $t$ is assigned the weight $w(t)=1$, the numbers $y_n$ count the number of distinct plane trees of size $n$, and the probability mass function $P_{\Phi}: \mathcal{T}_n \to [0,1]$ specifies the uniform probability distribution on $\mathcal{T}_n$.
\end{example}

\begin{example}\label{ex:dary}
The family of $d$-ary trees is obtained as the simply generated family of trees whose weight sequence $(\phi_k)_{k \geq 0}$ satisfies $\phi_m=\binom{d}{m}$ for every $m \geq 0$. This takes into account that there are $\binom{d}{m}$ many types of vertices of degree $m$ in $d$-ary trees. 
The weight $w(t)$ of a plane tree $t$ then equals the number of distinct $d$-ary trees with plane representation $t$ and the numbers $y_n$ count the number of distinct $d$-ary trees of size $n$. The uniform probability distribution on the set of $d$-ary trees of size $n$ thus induces the probability mass function $P_{\Phi}$ on $\mathcal{T}_n$ via the correspondence between $d$-ary trees and their plane representations.
\end{example}

\begin{example}\label{ex:motzkin}
The family of Motzkin trees is the family of ordered rooted trees such that each vertex has either zero, one or two children. In particular, we do not distinguish between left-unary and right-unary vertices as in the case of binary trees, i.e.,~there is only one type of unary vertices. The weight sequence $(\phi_k)_{k \geq 0}$ with $\phi_0=\phi_1=\phi_2=1$ and $\phi_k=0$ for $k \geq 3$ corresponds to the simply generated family of Motzkin trees, and the probability mass function $P_{\Phi}: \mathcal{T}_n \to [0,1]$ corresponds to the uniform probability distribution on the set of Motzkin trees of size $n$. 
\end{example}

\begin{example}\label{ex:labeled}
 Given an (unordered) labelled tree $t$, there are $\prod_{v \in t}\deg(v)!$ many possibilities to define an ordering on its vertices in order to obtain an ordered labelled tree, that is, $\prod_{v \in t}\deg(v)!$ many ordered labelled trees correspond to the same unordered labelled tree $t$. Furthermore, there are $n!$ many possibilities to label a plane tree of size $n$ in order to obtain an ordered labelled tree, that is, a plane tree of size $n$ corresponds to $n!$ ordered labelled trees.  The family of (unordered) labelled trees is obtained as the simply generated family of trees whose weight sequence $(\phi_k)_{k \geq 0}$ satisfies $\phi_k=1/k!$ for every $k \geq 0$: Thus, the weight of a plane tree $t$ equals $(\prod_{v \in t}\deg(v)!)^{-1}$, and the total weight $y_n$ of all plane trees of size $n$ equals $\frac{1}{n!}$ times the number of unordered labelled trees of size $n$. 
\end{example}

Closely related to the concept of simply generated families of trees is the concept of Galton--Watson processes:
Let $\xi$ be a non-negative integer-valued random variable (called an \emph{offspring distribution}). 
A \emph{Galton--Watson branching process} (see for example \cite{Janson12}) with offspring distribution $\xi$ generates a random plane tree $T$ as follows: in a top-down way, starting at the root vertex, we determine for each vertex $v$ of $T$ independently its degree $\deg(v)$ according to the distribution $\xi$. The probability that $\deg(v)=k$ for some integer $k$ is given by $\Prob(\xi=k)$. If $\deg(v)=k>0$, we attach $k$ new vertices to $v$ and the process continues at these newly attached vertices. If $\deg(v)=0$, the process stops at this vertex. It is thus convenient to assume that $\Prob(\xi=0)>0$. Note that this process might generate infinite trees with non-zero probability. 
We find that the probability $\nu(t)$ that a tree $t \in \mathcal{T}$ is generated by a Galton--Watson branching process with offspring distribution $\xi$ is
\begin{align*}
\nu(t)=\prod_{v \in t} \Prob(\xi=\deg(v))=\prod_{k \geq 0}\Prob(\xi=k)^{d_k(t)}.
\end{align*}
A random plane tree generated by a Galton--Watson process is called an \emph{unconditioned Galton--Watson tree}. Conditioning the Galton--Watson tree on the event that $|T|=n$, we obtain a probability mass function $P_{\xi}$ on the set $\mathcal{T}_n$ of plane trees
of size $n$ defined by
\begin{align*}
P_{\xi}(t)=\frac{\nu(t)}{\sum_{t' \in \mathcal{T}_n}\nu(t')}.
\end{align*}
A random variable taking values in $\mathcal{T}_n$ according to the probability mass function $P_{\xi}$ is called a \emph{conditioned Galton--Watson tree} of size $n$. 
A Galton--Watson process with offspring distribution $\xi$ that satisfies $\Ex(\xi)=1$ is called \emph{critical}. \\

Let $\mathcal{F}$ be a simply generated family of trees with weight distribution $(\phi_m)_{m \geq 0}$. In many cases, it is possible to view a random tree of size $n$ drawn from $\mathcal{T}_n$ according to the probability mass function $P_{\Phi}$ as a conditioned Galton--Watson tree (see for example \cite{Janson12}): 
let $R>0$ denote the radius of convergence of the series $\Phi(x)=\sum_{k \geq 0}\phi_kx^k$, and assume that there is $\tau \in (0,R]$ with $\tau\Phi'(\tau)=\Phi(\tau)$. Define an offspring distribution $\xi$ by
\begin{align}\label{eq:xi-according-to-phi}
\Prob(\xi=m)=\phi_m \tau^m\Phi(\tau)^{-1}
\end{align}
for every $m \geq 0$.
This is well-defined, as
\begin{align}\label{eq:xi-well-defined}
\sum_{m \geq 0}\Prob(\xi=m)=\sum_{m \geq 0}\frac{\phi_m \tau^m}{\Phi(\tau)}=\frac{\Phi(\tau)}{\Phi(\tau)}=1,
\end{align}
and furthermore, we have
\begin{align}\label{eq:xi-expectation}
\Ex(\xi)=\sum_{m \geq 0}m\Prob(\xi=m)=\sum_{m \geq 0}\frac{m\phi_m \tau^m}{\Phi(\tau)}=\frac{\tau \Phi'(\tau)}{\Phi(\tau)}=1.
\end{align}
Thus, $\xi$ is an offspring distribution of a critical Galton--Watson process. In particular, $\xi$ defined as in \eqref{eq:xi-according-to-phi} induces the same probability mass function on $\mathcal{T}_n$ as the weight sequence
$(\phi_m)_{m \geq 0}$, since we have 
\begin{align}\label{eq:xi-distributions-same}
P_{\xi}(t)=\frac{\nu(t)}{\sum_{t' \in \mathcal{T}_n}\nu(t')}
=\frac{\Phi(\tau)^n\tau^{n-1}\prod_{k \geq 0}(\phi_k)^{d_k(t)}}{\Phi(\tau)^n\tau^{n-1}\sum_{t' \in \mathcal{T}_n}\prod_{k \geq 0}(\phi_k)^{d_k(t')}}=\frac{w(t)}{\sum_{t' \in \mathcal{T}_n}w(t')}=P_{\Phi}(t).
\end{align}
Hence, many results proved in the context of Galton--Watson trees become applicable in the setting of simply generated families of trees. For a given simply generated family of trees $\mathcal{F}$ with weight sequence $(\phi_k)_{k \geq 0}$, we call the Galton--Watson process with offspring distribution $\xi$ defined as in \eqref{eq:xi-according-to-phi} the \emph{Galton--Watson process corresponding to $\mathcal{F}$}.
Regarding the variance of $\xi$, we find
\begin{align}\label{eq:xi-finite-variance}
\mathbb{V}(\xi)=\Ex(\xi^2)-\Ex(\xi)^2=\Ex(\xi(\xi-1))=\sum_{m \geq 0}m(m-1)\frac{\phi_m\tau^m}{\Phi(\tau)}=\frac{\tau^2\Phi''(\tau)}{\Phi(\tau)}.
\end{align}
Note that if $\tau < R$, then $\mathbb{V}(\xi) < \infty$, but if $\tau=R$, $\mathbb{V}(\xi)$ might be infinite. However, we will only consider weight sequences $(\phi_k)_{k\geq 0}$ for which the corresponding offspring distribution $\xi$ satisfies $\mathbb{V}(\xi)<\infty$.

\setcounter{example}{0}

\begin{example} (continued)
For the family of plane trees, we have $\Phi(x)=\sum_{k \geq 0}x^k$. We find that $\tau = 1/2$ solves the equation $\tau \Phi'(\tau)=\Phi(\tau)$. Thus, the offspring distribution $\xi$ of the Galton--Watson process corresponding to the family of plane trees is given by $\Prob(\xi=m)=2^{-m-1}$ for every $m \geq 0$ (a geometric distribution).
\end{example}

\begin{example} (continued)
For the family of $d$-ary trees, we find $\Phi(x)=(1+x)^d$ and $\tau = (d-1)^{-1}$. The offspring distribution $\xi$ of the Galton--Watson process corresponding to the family of $d$-ary trees is a binomial distribution with $\Prob(\xi=m)=\binom{d}{m}d^{-d}(d-1)^{d-m}$ for $0 \leq m \leq d$. 
\end{example}

\begin{example} (continued)
In the case of Motzkin trees, we have $\Phi(x)=1+x+x^2$ and $\tau=1$. The Galton--Watson process with offspring distribution $\xi$ defined by $\Prob(\xi=m)=1/3$ if $0 \leq m \leq 2$ and $\Prob(\xi=m)=0$ otherwise corresponds to the family of Motzkin trees.
\end{example}

\begin{example} (continued)
We obtain $\Phi(x)=e^x$ for the family of labelled trees. The equation $\tau \Phi'(\tau)=\Phi(\tau)$ is solved by $\tau = 1$ in this case. The Galton--Watson process corresponding to the family of labelled trees is thus defined by the offspring distribution $\xi$ with $\Prob(\xi=m)=(em!)^{-1}$ for every $m \geq 0$ (i.e., $\xi$ is a Poisson distribution).
\end{example}

Our first ingredient on the way to our main result is the following lemma on the total number of fringe subtrees of a given size in a conditioned Galton--Watson tree $T_n$ of size $n$:

\begin{lemma}\label{lemma:galton-expectation-variance}
Let $Z_{n,k}$ be the number of fringe subtrees of size $k$ in a conditioned Galton--Watson tree of size $n$ whose  offspring distribution $\xi$ satisfies $\Ex(\xi)=1$ and $\mathbb{V}(\xi)=\sigma^2<\infty$.
Then we have 
\begin{equation}\label{eq:exp_val_1}
\Ex(Z_{n,k}) =\frac{n}{\sqrt{2\pi\sigma^2}k^{3/2}}(1+o(1)),
\end{equation}
 and $\mathbb{V}(Z_{n,k}) = O(n/k^{3/2})$ uniformly in $k$ for $k \leq \sqrt{n}$ as $k,n \to \infty$. Moreover, for all $k \leq n$, we have 
\begin{equation}\label{eq:exp_val_2}
\Ex(Z_{n,k}) = O \Big( \frac{n^{3/2}}{k^{3/2}(n-k+1)^{1/2}} \Big).
\end{equation}
\end{lemma}
\begin{proof}
We make extensive use of the results in Janson's paper \cite{Janson16}. Let $S_n$ be the sum of $n$ independent copies of the offspring distribution: $S_n = \sum_{i = 1}^n \xi_i$. By \cite[Lemma 5.1]{Janson16}, we have
$$\Ex(Z_{n,k}) = \frac{\Prob(S_{n-k} = n-k)}{\Prob(S_n = n-1)} q_k n,$$
where $q_k$ is the probability that an \emph{unconditioned} Galton--Watson tree with offspring distribution $\xi$ has final size $k$. Moreover, by \cite[Lemma 5.2]{Janson16}, we have
$$\frac{\Prob(S_{n-k} = n-k)}{\Prob(S_n = n-1)} = 1 + O \Big ( \frac{k}{n} \Big) + o(n^{-1/2})$$
uniformly for all $k$ with $1 \leq k \leq \frac{n}{2}$ as $n \to \infty$, and by \cite[Eq.~(4.13)]{Janson16} (see also Kolchin \cite{Kolchin86}),
$$q_k \sim \frac{1}{\sqrt{2\pi \sigma^2}} k^{-3/2}$$
as $k \to \infty$. Combining the two, we obtain the desired asymptotic formula~\eqref{eq:exp_val_1} for $\Ex(Z_{n,k})$ if $k \leq \sqrt{n}$ and both $k$ and $n$ tend to infinity. For arbitrary $k$, \cite[Lemma 5.2]{Janson16} states that
$$\frac{\Prob(S_{n-k} = n-k)}{\Prob(S_n = n-1)} = O \Big( \frac{n^{1/2}}{(n-k+1)^{1/2}} \Big).$$
The estimate~\eqref{eq:exp_val_2} follows.

\medskip

For the variance, we can similarly employ \cite[Lemma 6.1]{Janson16}, which gives us
\begin{align*}
\mathbb{V}(Z_{n,k}) &= \frac{\Prob(S_{n-k} = n-k)}{\Prob(S_n = n-1)} q_k n  - \Big( \frac{\Prob(S_{n-k} = n-k)}{\Prob(S_n = n-1)} \Big)^2 q_k^2 n(2k-1) \\
&\quad + \Big( \frac{\Prob(S_{n-2k} = n-2k+1)}{\Prob(S_n = n-1)} - \Big( \frac{\Prob(S_{n-k} = n-k)}{\Prob(S_n = n-1)} \Big)^2 \Big) q_k^2n(n-2k+1).
\end{align*}
Finally, by \cite[Lemma 6.2]{Janson16},
$$\frac{\Prob(S_{n-2k} = n-2k+1)}{\Prob(S_n = n-1)} - \Big( \frac{\Prob(S_{n-k} = n-k)}{\Prob(S_n = n-1)} \Big)^2 = O \Big( \frac{1}{n} \Big)$$
for $k \leq \sqrt{n}$, uniformly in $k$. Combining all estimates yields $\mathbb{V}(Z_{n,k}) = O(q_kn) = O(n/k^{3/2})$, which completes the proof.
\end{proof}

From this, we can now derive the following lemma on fringe subtrees of a random tree $T_n$ of size $n$ drawn from a simply generated family $\mathcal{F}$:

\begin{lemma}\label{lemma:galton-watson-sk}
Let $T_n$ be a random tree of size $n$ drawn randomly from a simply generated family of trees $\mathcal{F}$ such that the offspring distribution $\xi$ of the corresponding critical Galton--Watson process satisfies $\mathbb{V}(\xi)=\sigma^2<\infty$. Let $a, \varepsilon$ be positive real numbers with $\varepsilon<\frac{1}{2}$. For every positive integer $k$ with $a \log n \leq k \leq n^{\varepsilon}$, let $\mathcal{S}_k \subseteq \mathcal{F}_k$ be a subset of trees of size $k$ from $\mathcal{F}$, and let $p_k$ be the probability that a random tree of size $k$ from the given family $\mathcal{F}$ belongs to $\mathcal{S}_k$. 
Now let $X_{n,k}$ denote the (random) number of fringe subtrees of size $k$ in the random tree $T_n$ which belong to $\mathcal{S}_k$. Moreover, let $Y_{n, \varepsilon}$ denote the (random) number of arbitrary fringe subtrees of size greater than $n^{\varepsilon}$ in $T_n$.
Then
\begin{itemize}
\item[(a)] $\Ex(X_{n,k})=p_kn(2\pi\sigma^2k^3)^{-1/2}(1+o(1))$, for all $k$ with $a\log n \leq k \leq n^{\varepsilon}$, the $o$-term being independent of $k$,
\item[(b)]$\mathbb{V}(X_{n,k})=O(p_kn/k^{3/2})$ for all $k$ with $a\log n \leq k \leq n^{\varepsilon}$, again with an $O$-constant independent of $k$,
\item[(c)]$\Ex(Y_{n, \varepsilon})=O(n^{1-\varepsilon/2})$, and
\item[(d)] with high probability, the following statements hold simultaneously:
\begin{itemize}
\item[(i)] $|X_{n,k}-\Ex(X_{n,k})|\leq p_k^{1/2}n^{1/2+\varepsilon}k^{-3/4}$ for all $k$ with $a \log k \leq k \leq n^{\varepsilon}$,
\item[(ii)]$Y_{n, \varepsilon}\leq n^{1-\varepsilon/3}$.
\end{itemize}
\end{itemize}
\end{lemma}
We emphasize (since it will be important later) that the inequality in part (d), item (i), does not only hold with high probability for each individual $k$, but that it is satisfied with high probability for all $k$ in the given range simultaneously.

\begin{proof}
Let $Z_{n,k}$ again denote the number of fringe subtrees of size $k$ in the conditioned Galton--Watson tree of size $n$ with offspring distribution $\xi$. Then $Z_{n,k}$ and the random number of fringe subtrees of size $k$ in a random tree $T_n$ of size $n$ drawn randomly from the simply generated family $\mathcal{F}$ are identically distributed. Furthermore, if a random tree $T_n$ of size $n$ drawn from $\mathcal{F}$ contains $Z_{n,k}$ many fringe subtrees of size $k$, then these fringe subtrees are again independent random trees in $\mathcal{F}_k$ with the same distribution.
Thus, $X_{n,k}$ can be regarded as a sum of $Z_{n,k}$ many Bernoulli random variables with probability $p_k$. We thus have (see \cite[Theorem 15.1]{Gut2005})
\begin{align*}
\Ex(X_{n,k})=p_k\Ex(Z_{n,k})=\frac{n p_k}{\sqrt{2 \pi \sigma^2}k^{3/2}}(1+o(1)),
\end{align*}
as well as
\begin{align*}
\mathbb{V}(X_{n,k})&=p_k^2\mathbb{V}(Z_{n,k})+p_k(1-p_k)\Ex(Z_{n,k})= O\left(\frac{np_k}{k^{3/2}}\right)
\end{align*}
by Lemma~\ref{lemma:galton-expectation-variance}, which proves part (a) and part (b). 
For part (c), we observe that
\begin{align*}
\Ex(Y_{n, \varepsilon})=\sum_{k >n^{\varepsilon}}\Ex(Z_{n,k})=O\left(n^{1-\varepsilon/2}\right),
\end{align*}
again by Lemma~\ref{lemma:galton-expectation-variance}.
In order to show part (d), we apply Chebyshev's inequality to obtain concentration on $X_{n,k}$: 
\begin{align*}
\Prob\left(|X_{n,k}-\Ex(X_{n,k})|\geq p_k^{1/2}n^{1/2+\varepsilon}k^{-3/4}\right)\leq \frac{\mathbb{V}(X_{n,k})}{p_kn^{1+2\varepsilon}k^{-3/2}}=O(n^{-2\varepsilon}).
\end{align*}
Hence, by the union bound, the probability that the stated inequality fails for any $k$ in the given range is only $O(n^{-\varepsilon})$, proving that the first statement holds with high probability. Finally, Markov's inequality implies that
\begin{align*}
\Prob\left(Y_{n, \varepsilon}>n^{1-\varepsilon/3}\right)\leq \frac{\Ex(Y_{n, \varepsilon})}{n^{1-\varepsilon/3}}=O(n^{-\varepsilon/6}),
\end{align*}
showing that the second inequality holds with high probability as well. 
\end{proof}

\subsection{Families of increasing trees}\label{sec:increasing}

An increasing tree is a rooted tree whose vertices are labelled $1,2,\ldots,n$ in such a way that the labels along any path from the root to a leaf are increasing. If one assigns a weight function to these trees in the same way as for simply generated trees, one obtains a simple variety of increasing trees.
The exponential generating function for the total weight satisfies the differential equation
\begin{equation}\label{de:inctrees}
Y'(x) = \Phi(Y(x)).
\end{equation}
A general treatment of simple varieties of increasing trees was given by Bergeron, Flajolet and Salvy in \cite{BergeronFS92}. Three special cases are of particular interest, as random elements from these families can be generated by a simple growth process. These are:
\begin{itemize}
\item recursive trees, where $\Phi(t) = e^t$;
\item generalized plane-oriented recursive trees (gports), where $\Phi(t) = (1-t)^{-r}$;
\item $d$-ary increasing trees, where $\Phi(t) = (1+t)^d$.
\end{itemize}
They are the increasing tree analogues of labelled trees, (generalized) plane trees and $d$-ary trees, respectively.
Collectively, these are sometimes called \emph{very simple families of increasing trees} \cite{Panholzer2007level}. In all these cases, the differential equation~\eqref{de:inctrees} has a simple explicit solution, namely
\begin{itemize}
\item $Y(x) = - \log(1-x)$ for recursive trees,
\item $Y(x) = 1 - (1-(r+1)x)^{1/(r+1)}$ for gports,
\item $Y(x) = (1-(d-1)x)^{-1/(d-1)} - 1$ for $d$-ary increasing trees.
\end{itemize}
It follows that the number (total weight, in the case of gports) of trees with $n$ vertices is
\begin{itemize}
\item $(n-1)!$ for recursive trees,
\item $\prod_{k=1}^{n-1} (k(r+1)-1)$ for gports,
\item $\prod_{k=1}^{n-1} (1+k(d-1))$ for $d$-ary increasing trees (in particular, $n!$ for binary increasing trees).
\end{itemize}
There is a natural growth process to generate these trees randomly: start with the root, which is labelled $1$. The $n$-th vertex (labelled $n$) is attached at random to one of the previous $n-1$ vertices, with a probability that is proportional to a linear function of the (out-)degree. Specifically, setting $\alpha = 0$ for recursive trees, $\alpha = 1/r$ for gports and $\alpha = -1/d$ for $d$-ary increasing trees, the probability to attach to a vertex $v$ with degree (number of children) $\ell$ is always proportional to $1 + \alpha \ell$. So in particular, all vertices are equally likely for recursive trees, vertices can only have up to $d$ children in $d$-ary increasing trees (since then the probability to attach further vertices becomes $0$), and vertices in generalized plane-oriented trees have a higher probability to become parent of a new vertex if they already have many children; hence they are also called preferential attachment trees.

It is well known that the special case $d = 2$ of $d$-ary increasing trees leads to a model of random binary trees that is equivalent to binary search trees, see for example \cite{Drmota09}.

We make use of known results on the total number of fringe subtrees of a given size in very simple families of increasing trees. In particular, we have the following formulas for the mean and variance (see \cite{Fuchs2012}):

\begin{lemma}\label{lemma:increasing-expectation-variance}
Consider a very simple family of increasing trees, and let $\alpha$ be defined as above. For every $k < n$, let $Z_{n,k}$ be the random number of fringe subtrees of size $k$ in a random tree of size $n$ drawn from the simple family of increasing trees. Then the expectation of $Z_{n,k}$ satisfies
\begin{align*}
\Ex(Z_{n,k})= \frac{(1+\alpha)n-\alpha}{((1+\alpha)k+1)((1+\alpha)k-\alpha)},
\end{align*}
and for the variance of $Z_{n,k}$, we have $\mathbb{V}(Z_{n,k})=O(n/k^2)$ uniformly in $n$ and $k$.

\end{lemma}

Now we obtain the following analogue of Lemma~\ref{lemma:galton-watson-sk}. The key difference is the asymptotic behaviour of the number of fringe subtrees with $k$ vertices as $k$ increases: instead of a factor $k^{-3/2}$, we have a factor $k^{-2}$.

\begin{lemma}\label{lemma:increasing-sk}
Let $T_n$ be a random tree of size $n$ drawn from a very simple family of increasing trees with $\alpha$ defined as above.
Let $a, \varepsilon$ be positive real numbers with $\varepsilon < \frac{1}{2}$. For every positive integer $k$ with $a \log n \leq k \leq n^{\varepsilon}$, let $\mathcal{S}_k$ be a subset of the possible shapes of a tree of size $k$, and let $p_k$ be the probability that a random tree of size $k$ from the given family has a shape that belongs to $\mathcal{S}_k$.
Now let $X_{n,k}$ denote the (random) number of fringe subtrees of size $k$ in the random tree $T_n$ whose shape belongs to $\mathcal{S}_k$. Moreover, let $Y_{n, \varepsilon}$ denote the (random) number of arbitrary fringe subtrees of size greater than $n^{\varepsilon}$ in $T_n$. Then
\begin{itemize}
\item[(a)] $\Ex(X_{n,k})= \frac{np_k}{(1+\alpha)k^2}(1+O(1/k))$ for all $k$ with $a \log n \leq k \leq n^{\varepsilon}$, the $O$-constant being independent of $k$,
\item[(b)] $\mathbb{V}(X_{n,k})=O(p_k n/k^2)$ for all $k$ with $a \log n \leq k \leq n^{\varepsilon}$, again with an $O$-constant being independent of $k$,
\item[(c)] $\Ex(Y_{n,\varepsilon})=O(n^{1-\varepsilon})$, and
\item[(d)] with high probability, the following statements hold simultaneously:
\begin{itemize}
\item[(i)] $|X_{n,k}-\Ex(X_{n,k})|\leq p_k^{1/2}k^{-1}n^{1/2+\varepsilon}$ for all $k$ with $a \log k \leq k \leq n^{\varepsilon}$,
\item[(ii)]$Y_{n, \varepsilon}\leq n^{1-\varepsilon/2}$.
\end{itemize}
\end{itemize}
\end{lemma}

\begin{proof}
The proof is similar to the proof of Lemma~\ref{lemma:galton-watson-sk}.
Again we find that $X_{n,k}$ can be regarded as a sum of $Z_{n,k}$ Bernoulli random variables with probability $p_k$. By \cite[Theorem 15.1]{Gut2005}, we have
\begin{align*}
\Ex(X_{n,k}) = p_k \Ex(Z_{n,k})
\end{align*}
as well as
\begin{align*}
\mathbb{V}(X_{n,k}) = p_k^2 \mathbb{V}(Z_{n,k}) + p_k(1-p_k) \Ex(Z_{n,k}).
\end{align*}
Now (a) and (b) both follow easily from Lemma~\ref{lemma:increasing-expectation-variance}.

In order to estimate $\Ex(Y_{n, \varepsilon})$, observe again that
\begin{align*}
\Ex(Y_{n, \varepsilon})=\sum_{k>n^{\varepsilon}}\Ex(Z_{n,k}).
\end{align*}
Now (c) also follows easily from Lemma~\ref{lemma:increasing-expectation-variance}. Finally, (d) is obtained from (b) and (c) by applying the Markov inequality, the Chebyshev inequality and the union bound in the same way as in the proof of Lemma~\ref{lemma:galton-watson-sk}.
\end{proof}

\section{Proof of Theorem~\ref{thm:master-theorem-simply-generated} and Theorem~\ref{thm:master-theorem-increasing}}

We will focus on the proof of Theorem~\ref{thm:master-theorem-simply-generated}, which is presented in two parts. First, the upper bound is verified; then we prove the lower bound, which has the same order of magnitude. A basic variant of the proof technique was already applied in the proof of Theorem 3.1 in \cite{RalaivaosaonaW15}.

\subsection{The upper bound}

For some integer $k_0$ (to be specified later), we can clearly bound the total number of isomorphism classes covered by the fringe subtrees of a random tree $T_n$ of size $n$ from above by the sum of
\begin{itemize}
\item[(i)] the total number of isomorphism classes of trees of size smaller than $k_0$, which is $\sum_{k < k_0} |\mathcal{I}_k|$ (a deterministic quantity that does not depend on the tree $T_n$), and
\item[(ii)] the total number of fringe subtrees of $T_n$ of size greater than or equal to $k_0$.
\end{itemize}
To estimate the number (i) of isomorphism classes of trees of size smaller than $k_0$, we note that $|\mathcal{I}_k| \leq e^{C_1k + o(k)}$ by condition (C1), thus also
\begin{align*}
\sum_{k < k_0} |\mathcal{I}_k| \leq e^{C_1k_0 + o(k_0)}.
\end{align*}
We can therefore choose $k_0 = k_0(n)$ for every $n$ in such a way that $k_0 = \frac{\log n}{C_1} - o(\log n)$ and
\begin{align*}
\sum_{k < k_0} |\mathcal{I}_k| = o\Big(\frac{n}{\sqrt{\log n}}\Big),
\end{align*}
thus making this part negligible. In order to estimate the number (ii) of fringe subtrees of $T_n$ of size greater than or equal to $k_0$, we apply Lemma~\ref{lemma:galton-watson-sk} with $\varepsilon=1/6$. We let $\mathcal{S}_k$ be the set of all trees of size $k$ generated by our simply generated family of trees, so that $p_k=1$, to obtain the upper bound
\begin{align*}
\sum_{k_0\leq k\leq n^{\varepsilon}}X_{n,k}+Y_{n, \varepsilon}&=\frac{n}{\sqrt{2\pi\sigma^2}}\sum_{k_0 \leq k \leq n^{\varepsilon}}\frac{1}{k^{3/2}}\left(1+o(1)\right)+O\left(n^{1-\varepsilon/3}\right)\\
&=\frac{2}{\sqrt{2\pi\sigma^2}}\frac{n}{\sqrt{k_0}}+o\left(\frac{n}{\sqrt{\log n}}\right),
\end{align*}
in expectation and with high probability as well, as the estimate from Lemma~\ref{lemma:galton-watson-sk} (part (d)) holds with high probability simultaneously for all $k$ in the given range. Now we combine the two bounds to obtain the upper bound on $A_n$ stated in Theorem~\ref{thm:master-theorem-simply-generated}, both in expectation and with high probability.

\subsection{The lower bound}

Let $\mathcal{S}_k$ now be the set of trees that belong to isomorphism classes in $\mathcal{J}_k$ (see condition (C2)). Our lower bound is based on counting only fringe subtrees which belong to $\mathcal{S}_k$ for suitable $k$. By condition (C2a), we know that the probability $p_k$ that a random tree in $\mathcal{F}$ belongs to a class in $\mathcal{J}_k$ tends to $1$ as $k \to \infty$.
Hence, by Lemma~\ref{lemma:galton-watson-sk}, we find that the number of fringe subtrees of size $k$ in $T_n$ that belong to $S_k$ is
\begin{align*}
X_{n,k} = \frac{n}{\sqrt{2\pi\sigma^2k^3}}(1 + o(1)),
\end{align*}
both in expectation and with high probability.

We show that most of these trees are the only representatives of their isomorphism classes as fringe subtrees. We choose a cut-off point $k_1 = k_1(n)$; the precise choice will be described later. For $k \geq k_1$, let $X_{n,k}^{(2)}$ denote the (random) number of unordered pairs of isomorphic trees (trees belonging to the same isomorphism class) among the fringe subtrees of size $k$ which belong to $\mathcal{S}_k$. We will determine an upper bound for its expected value.

To this end, let $\ell$ denote the number of isomorphism classes of trees in $\mathcal{S}_k$, and let $q_1,q_2,\ldots,q_{\ell}$ be the probabilities that a random tree of size $k$ lies in the respective classes.
By condition (C2b), we have $q_i \leq e^{-C_2k + o(k)}$ for every $i$. Let us condition on the event that $X_{n,k}=N$ for some integer $0 \leq N \leq n$. Those $N$ fringe subtrees are all independent random trees. Thus, for each of the $\binom{N}{2}$ pairs of fringe subtrees, the probability that both belong to the $i$-th isomorphism class is $q_i^2$. This gives us

\begin{align*}
\Ex(X_{n,k}^{(2)} \mid X_{n,k}=N)= \binom{N}{2}\sum_{i=1}^{\ell} q_i^2\leq \frac{n^2}{2}\sum_{i=1}^{\ell} q_i e^{-C_2k + o(k)} \leq \frac{n^2}{2} e^{-C_2k + o(k)}.
\end{align*}
Since this holds for all $N$, the law of total expectation yields 
\begin{align*}
\Ex(X_{n,k}^{(2)})\leq \frac{n^2}{2} e^{-C_2k + o(k)}.
\end{align*}
Summing over $k \geq k_1$, we find that

\begin{align*}
\sum_{k \geq k_1}\Ex(X_{n,k}^{(2)})\leq \frac{n^2}{2} \sum_{k \geq k_1} e^{-C_2k + o(k)} \leq \frac{n^2}{2} e^{-C_2k_1 + o(k_1)}.
\end{align*}
We can therefore choose $k_1$ in such a way that $k_1 = \frac{\log n}{C_2} - o(\log n)$ and
\begin{align*}
\sum_{k \geq k_1}\Ex(X_{n,k}^{(2)}) = o \Big( \frac{n}{\sqrt{\log n}} \Big).
\end{align*}
If an isomorphism class of trees of size $k$ occurs $m$ times among the fringe subtrees of a random tree of size $n$, it contributes $m-\binom{m}{2}$ to the random variable $X_{n,k}-X_{n,k}^{(2)}$. As $m-\binom{m}{2}\leq 1$ for all non-negative integers $m$, we find that $X_{n,k}-X_{n,k}^{(2)}$ is a lower bound on the total number of isomorphism classes covered by fringe subtrees of $T_n$. This gives us
\begin{align*}
A_n \geq \sum_{k_1 \leq k \leq n^{\varepsilon}}X_{n,k}-\sum_{k_1 \leq k \leq n^{\varepsilon}}X_{n,k}^{(2)},
\end{align*}
where we choose $\varepsilon$ as in the proof of the upper bound. The second sum is negligible since it is $o(n/\sqrt{\log n})$ in expectation and thus also with high probability by the Markov inequality. For the first sum, the same calculation as for the upper bound (using Lemma~\ref{lemma:galton-watson-sk}) shows that it is 
\begin{align*}
 \frac{2n}{\sqrt{2\pi\sigma^2 k_1}}+o\left(\frac{n}{\sqrt{\log n}}\right)
\end{align*} 
both in expectation and with high probability. This yields the desired statement.

\subsection{Increasing trees}

With Lemma~\ref{lemma:increasing-sk} in mind, it is easy to see that the proof of Theorem~\ref{thm:master-theorem-increasing} is completely analogous. The only difference is that sums of the form $\sum_{a \leq k \leq b} k^{-3/2}$ become sums of the form $\sum_{a \leq k \leq b} k^{-2}$.

As the main idea of these proofs is to split the number of distinct fringe subtrees into the number of distinct fringe subtrees of size at most $k$ plus the number of distinct fringe subtrees of size greater than $k$ for some suitably chosen integer $k$, this type of argument is called a cut-point argument and the integer $k$ is called the cut-point (see \cite{FlajoletGM97}). This basic technique is applied in several previous papers to similar problems (see for instance \cite{Devroye98}, \cite{FlajoletGM97}, \cite{RalaivaosaonaW15}, \cite{SeelbachLo18}).

\section{Applications: simply generated trees}

Let $\mathcal{F}$ be a simply generated family of trees, such that the corresponding critical Galton--Watson process with offspring distribution $\xi$ satisfies $\mathbb{V}(\xi)<\infty$.
In this section, we show that Theorem~\ref{thm:master-theorem-simply-generated} can be used to count the numbers
\begin{itemize}
\item[(i)] $H_n$ of distinct trees (as members of $\mathcal{F}$),
\item[(ii)] $J_n$ of distinct plane trees, and
\item[(iii)] $K_n$ of distinct unordered trees
\end{itemize}
represented by the fringe subtrees of a random tree $T_n$ of size $n$ drawn randomly from the family $\mathcal{F}$. In order to estimate the numbers $J_n$ and $K_n$, we additionally need a result by Janson \cite{Janson16} on additive functionals in conditioned Galton--Watson trees:
Let $f: \mathcal{T} \to \mathbb{R}$ denote a function mapping a plane tree to a real number (called a \emph{toll-function}). We define a mapping $F: \mathcal{T} \to \mathbb{R}$ by
\begin{align*}
F(t)=\sum_{v \in t}f(t(v)). 
\end{align*}
Such a mapping $F$ is then called an \emph{additive functional}. Equivalently, $F$ can be defined by a recursion. If $t_1,t_2,\ldots,t_h$ are the root branches of $t$ (the components resulting when the root is removed), then
$$F(t) = f(t) + \sum_{j=1}^h F(t_j).$$
The following theorem follows from Theorem 1.3 and Remark 5.3 in \cite{Janson16}:
\begin{theo}[\hspace{1sp}\cite{Janson16}, Theorem 1.3 and Remark 5.3]\label{thm:janson}
Let $T_n$ be a conditioned Galton--Watson tree of size $n$, defined by an offspring distribution $\xi$ with $\Ex(\xi)=1$, and let $T$ be the corresponding unconditioned Galton--Watson tree. If $\Ex(|f(T)|)<\infty$ and $|\Ex(f(T_k))| = o(k^{1/2})$, then 
\begin{align*}
\frac{F(T_n)}{n}\overset{p}{\to} \Ex(f(T)),
\end{align*}
that is,
$$\lim_{n \to \infty} \Prob \Big( \Big| \frac{F(T_n)}{n} - \Ex(f(T)) \Big| > \varepsilon \Big) = 0$$
for every $\varepsilon > 0$.
\end{theo}

\subsection{Distinct fringe subtrees in simply generated trees}\label{subsec:distinct}

In order to count distinct fringe subtrees in a random tree $T_n$ of size $n$ drawn from a simply generated family of trees $\mathcal{F}$, we consider two trees as isomorphic if they are identical as members of $\mathcal{F}$ and verify that the conditions of Theorem~\ref{thm:master-theorem-simply-generated} are satisfied. That is, we consider a partition of $\mathcal{F}_k$ into isomorphism classes of size one, or in other words, each tree is isomorphic only to itself. 
The total number of isomorphism classes $|\mathcal{I}_k|$ is thus the total number of trees in $\mathcal{F}$ of size $k$. In order to ensure that condition (C1) from Theorem~\ref{thm:master-theorem-simply-generated} is satisfied, we need to make an additional assumption on $\mathcal{F}$:
We assume that the weights $\phi_k$ of the weight sequence $(\phi_k)_{k \geq 0}$ are integers, and that each tree $t \in \mathcal{F}$ corresponds to a weight of one unit, such that the total weight $y_n$ of all plane trees of size $n$ then equals the number of distinct trees of size $n$ in our simply generated family $\mathcal{F}$ of trees. This assumption is e.g.~satisfied by the simply generated family of plane trees (Example~\ref{ex:plane}), the family of $d$-ary trees (Example~\ref{ex:dary}) and the family of Motzkin trees (Example~\ref{ex:motzkin}).
We have the following theorem on the asymptotic growth of the numbers $y_n$:

\begin{theo}[see \cite{Drmota09}, Theorem 3.6 and Remark 3.7]\label{thm:number-of-trees}
Let $R$ be the radius of convergence of $\Phi(x)=\sum_{m\geq 0}\phi_mx^m$ and suppose that there exists $\tau \in (0,R]$ with $\tau\Phi'(\tau)=\Phi(\tau)$. Let $d$ be the greatest common divisor of all indices $m$ with $\phi_m>0$. Then
\begin{align*}
y_n = d\sqrt{\frac{\Phi(\tau)}{2\pi \Phi''(\tau)}}\frac{\Phi'(\tau)^n}{n^{3/2}}\left(1+O(n^{-1})\right),
\end{align*}
if $n \equiv 1 \mod d$, and $y_n=0$ if $n \not\equiv 1 \mod d$.
\end{theo}
For the sake of simplicity, we will tacitly assume that $d=1$ holds for the simply generated families of trees considered below, though all results presented below can be easily shown to hold for $d \neq 1$ as well. 
We obtain the following result from Theorem~\ref{thm:master-theorem-simply-generated} regarding the number of distinct fringe subtrees in a random tree $T_n$ of size $n$ drawn randomly from a simply generated family of trees whose weight sequence is a sequence of integers:

\begin{theo}\label{thm:ordered-simply-generated}
Let $H_n$ denote the total number of distinct fringe subtrees in a random tree $T_n$ of size $n$ from a simply generated family $\mathcal{F}$ of trees with generating series $\Phi(x)=\sum_{m \geq 0}\phi_mx^m$, whose weights $\phi_m$ are integers. Let $R$ denote the radius of convergence of $\Phi$ and suppose that there exists $\tau \in (0,R]$ with $\tau\Phi'(\tau)=\Phi(\tau)$. Furthermore, suppose that the variance of the offspring distribution $\xi$ of the Galton--Watson process corresponding to $\mathcal{F}$ satisfies $\mathbb{V}(\xi)=\sigma^2<\infty$.
Then for $c=2\tau^{-1}(\Phi(\tau)\log(\Phi'(\tau)))^{1/2}(2\pi\Phi''(\tau))^{-1/2}$, we have
\begin{itemize}
\item[(i)] $\displaystyle \Ex(H_n)=c\frac{n}{\sqrt{\log n}}(1+o(1))$,
\item[(ii)] $\displaystyle H_n=c\frac{n}{\sqrt{\log n}}(1+o(1))$ with high probability.
\end{itemize}
\end{theo}
The first part (i) of Theorem~\ref{thm:ordered-simply-generated} was already shown in \cite{FlajoletSS90,RalaivaosaonaW15}, our new contribution is part (ii). 
\begin{proof}
We verify that the conditions of Theorem~\ref{thm:master-theorem-simply-generated} are satisfied
if we consider the partition of $\mathcal{F}$ into isomorphism classes of size one, that is, each tree $t$ is isomorphic only to itself. We find that
\begin{align*}
|\mathcal{I}_k|=y_k,
\end{align*}
i.e.,~the number $|\mathcal{I}_k|$ of isomorphism classes of trees of size $k$ equals the number $y_k$ of distinct trees of size $k$ in the respective simply generated family of trees $\mathcal{F}$.
With Theorem~\ref{thm:number-of-trees}, we have
\begin{align*}
|\mathcal{I}_k|=\sqrt{\frac{\Phi(\tau)}{2\pi\Phi''(\tau)}}\frac{\Phi'(\tau)^k}{k^{3/2}}(1+O(k^{-1})),
\end{align*}
so condition (C1) is satisfied with $C_1=\log(\Phi'(\tau))$.
In order to show that condition (C2) holds, define $\mathcal{J}_k
=\mathcal{I}_k$, so that every random tree of size $k$ in the family $\mathcal{F}$ belongs to a class in $\mathcal{J}_k$, and the probability that a random tree in $\mathcal{F}$ of size $k$ lies in a fixed isomorphism class $I \in \mathcal{J}_k$ is $1/y_k$. Thus, condition (C2) holds as well, and we have $C_2= C_1 = \log(\Phi'(\tau))$.
Recall that by \eqref{eq:xi-finite-variance}, we find that the variance of the Galton--Watson process corresponding to $\mathcal{F}$ is given by
\begin{align*}
\mathbb{V}(\xi)=\sigma^2=\frac{\tau^2\Phi''(\tau)}{\Phi(\tau)}.
\end{align*}
Theorem~\ref{thm:ordered-simply-generated} now follows directly from Theorem~\ref{thm:master-theorem-simply-generated}.
\end{proof}

The following results follow as special cases of Theorem~\ref{thm:ordered-simply-generated} for particular simply generated families of trees:
\begin{corollary}\label{cor:distinct-plane}
Let $H_n$ denote the total number of distinct fringe subtrees in a uniformly random plane tree of size $n$. Then
\begin{itemize}
\item[(i)] $\displaystyle \Ex(H_n)=\sqrt{\frac{\log 4}{\pi}}\frac{n}{\sqrt{\log n}}(1+o(1))$,
\item[(ii)] $\displaystyle H_n=\sqrt{\frac{\log 4}{\pi}}\frac{n}{\sqrt{\log n}}(1+o(1))$ with high probability.
\end{itemize}
\end{corollary}
\begin{proof}
The family of plane trees is obtained as the simply generated family of trees with weight sequence $(\phi_k)_{k \geq 0}$ with $\phi_k=1$ for every $k \geq 0$ (see Example~\ref{ex:plane}). In particular, we find that $\Phi(x)=\sum_{k\geq 0}x^k = \frac{1}{1-x}$ and that $\tau=\frac{1}{2}$ solves the equation $\tau\Phi'(\tau)=\Phi(\tau)$. Thus, the constant $c$ from Theorem~\ref{thm:ordered-simply-generated} evaluates to
\begin{align*}
c=\frac{2}{\tau}\sqrt{\frac{\Phi(\tau)\log(\Phi'(\tau))}{2\pi\Phi''(\tau)}}=\sqrt{\frac{\log 4}{\pi}}.
\end{align*}

\end{proof}

For $d$-ary trees, we obtain the following corollary (the result for binary trees was already given in the conference version \cite{SeelbachWagner20} of this paper):
\begin{corollary}\label{cor:distinct-binary}
Let $H_n$ denote the total number of distinct fringe subtrees in a uniformly random $d$-ary tree of size $n$. Then
\begin{itemize}
\item[(i)] $\displaystyle \Ex(H_n)=\left(\frac{2d}{\pi} \Big( \frac{d}{d-1} \log d - \log(d-1)\Big)\right)^{1/2}\frac{n}{\sqrt{\log n}}(1+o(1))$,
\item[(ii)] $\displaystyle H_n=\left(\frac{2d}{\pi} \Big( \frac{d}{d-1} \log d - \log(d-1)\Big)\right)^{1/2}\frac{n}{\sqrt{\log n}}(1+o(1))$ with high probability.
\end{itemize}
In particular, for the family of binary trees, we obtain 
$$H_n=2\sqrt{\frac{\log 4}{\pi}}\cdot \frac{n}{\sqrt{\log n}}(1+o(1)),$$
both in expectation and with high probability.
\end{corollary}
\begin{proof}
The family of $d$-ary trees is obtained as the simply generated family of trees with weight sequence $(\phi_k)_{k \geq 0}$, where $\phi_k=\binom{d}{k}$ for every $k \geq 0$ (see Example~\ref{ex:dary}). We find that $\Phi(x)=(1+x)^d$ and that $\tau=(d-1)^{-1}$ satisfies the equation $\tau\Phi'(\tau)=\Phi(\tau)$. Therefore, the constant $c$ in Theorem~\ref{thm:ordered-simply-generated} evaluates for the case of $d$-ary trees to
\begin{align*}
c=\frac{2}{\tau}\sqrt{\frac{\Phi(\tau)\log(\Phi'(\tau))}{2\pi\Phi''(\tau)}}=\left(\frac{2d}{\pi} \Big( \frac{d}{d-1} \log d - \log(d-1)\Big)\right)^{1/2}.
\end{align*}
\end{proof}

We remark that Theorem~\ref{thm:ordered-simply-generated} does not apply to the family of labelled trees (see Example~\ref{ex:labeled}), as the weight sequence corresponding to the family of labelled trees is not a sequence of integers. In particular, the number of labelled trees of size $n$ is $n^{n-1}$ (see for example \cite{Drmota09}), and thus, a partition of the set of labelled trees of size $n$ into isomorphism classes of size one does not satisfy condition (C1) from Theorem~\ref{thm:master-theorem-simply-generated}.
The total number $L_n$ of distinct fringe subtrees in a uniformly random labelled tree of size $n$ was estimated in \cite{RalaivaosaonaW15}, where it was shown that 
\begin{align*}
\Ex(L_n)=\sqrt{\frac{2}{\pi}}\frac{n\sqrt{\log\log n}}{\sqrt{\log n}}\left(1+O\left(\frac{\log \log \log n}{\log \log n}\right)\right).
\end{align*}
Here, two fringe subtrees are considered the same if there is an isomorphism that preserves the relative order of the labels. 

\subsection{Distinct plane fringe subtrees in simply generated trees}\label{subsec:distinctplane}

In this subsection, we consider simply generated families $\mathcal{F}$ of trees which admit a plane embedding: 
For instance, for the family of $d$-ary trees (see Example~\ref{ex:dary}), we find that each $d$-ary tree can be considered as a plane tree in a natural way by simply forgetting the positions to which the branches of the vertices are attached, such that there is no distinction between different types of vertices of the same degree. Likewise,  trees from the simply generated family of labelled trees (see Example~\ref{ex:labeled}) admit a unique plane representation if we order the children of each vertex according to their labels and then disregard the vertex labels.
For the family of plane trees (see Example~\ref{ex:plane}), the results from this section will be equivalent to the results presented in the previous section.

In order to count the number of distinct plane trees represented by the fringe subtrees of a random tree $T_n$ drawn from a simply generated family of trees, we need the following result which follows from Theorem~\ref{thm:janson}:

\begin{lemma}\label{lem:plane-janson}
Let $\xi$ be the offspring distribution of a critical Galton--Watson process satisfying $\mathbb{V}(\xi)=\sigma^2<\infty$, and let $T_k$ be a conditioned Galton--Watson tree of size $k$ with respect to $\xi$. Let $M=\{m \in \mathbb{N} \mid \Prob(\xi=m)>0\}$, and let
\begin{align*}
\mu =\sum_{m \in M}\Prob(\xi=m)\log(\Prob(\xi=m)).
\end{align*}
Furthermore, let 
\begin{align*}
\nu(T_k)= \prod_{v \in T_k}\mathbb{P}(\xi=\deg(v))
\end{align*}
(as defined in Section~\ref{sec:simply-generated}). The probability that
\begin{align*}
\nu(T_k)\leq e^{(\mu+\varepsilon)k}
\end{align*}
holds tends to $1$ for every fixed $\varepsilon>0$ as $k \to \infty$. 
\end{lemma}
\begin{proof}
Let $\rho(t)$ denote the degree of the root vertex of a plane tree $t \in \mathcal{T}$, and define the function $f:\mathcal{T}\to \mathbb{R}$ by
\begin{align*}
f(t)=\begin{cases}
\log(\Prob(\xi=\rho(t))) \quad &\text{ if } \Prob(\xi=\rho(t))>0,\\
0 &\text{otherwise.}
\end{cases}
\end{align*}
For every $t \in \mathcal{T}$ with $\nu(t)>0$, the associated additive functional is
\begin{align*}
F(t)=\sum_{v \in t}f(t(v))=\sum_{v \in t}\log\left(\Prob(\xi=\rho(t(v)))\right)=\log\left(\prod_{v \in t}\Prob(\xi=\deg(v))\right)=\log(\nu(t)).
\end{align*}
Let $T$ denote the unconditioned Galton--Watson tree corresponding to $\xi$. Then
\begin{align*}
\Ex(|f(T)|)=\sum_{m \in M}\Prob(\xi=m)|\log(\Prob(\xi=m))|.
\end{align*}
Note that if $\Prob(\xi=m)>e^{-m}$, we have $|\log(\Prob(\xi=m))|\leq m$, and if $\Prob(\xi=m)\leq e^{-m}$, we have $\Prob(\xi=m)|\log(\Prob(\xi=m))|\leq e^{-m/2}$. Thus, we are able to bound $\Ex(|f(T)|)$ from above by
\begin{align}\label{eq:expectation-f-finite}
\Ex(|f(T)|)\leq \sum_{m \geq 0}\Prob(\xi=m)m+\sum_{m \geq 0}e^{-m/2}=\Ex(\xi)+\frac{\sqrt{e}}{\sqrt{e}-1}<\infty,
\end{align}
as the Galton--Watson process is critical by assumption. Furthermore, we have
\begin{align*}
|\Ex(f(T_k))|\leq \sum_{m \geq 0}\Prob(\rho(T_k)=m)|\log(\Prob(\xi=m))|.
\end{align*}
By (2.7) in \cite{janson05}, there is a constant $c>0$ (independent of $k$ and $m$) such that
\begin{align*}
\Prob(\rho(T_k)=m)\leq cm\Prob(\xi=m)
\end{align*}
for all $m,k\geq 0$. We thus find
\begin{align}\label{eq:toll-function-finite}
|\Ex(f(T_k))|\leq c\sum_{m \in M}m\Prob(\xi=m)|\log(\Prob(\xi=m))|\leq c\sum_{m \geq 0}\Prob(\xi=m)m^2 + c\sum_{m \geq 0}me^{-m/2}<\infty,
\end{align}
as $\mathbb{V}(\xi)<\infty$ by assumption. As the upper bound holds independently of $k$, we thus have $|\Ex(f(T_k))| = O(1)$. Altogether, we find that the requirements of Theorem~\ref{thm:janson} are satisfied. Let
\begin{align*}
\mu=\Ex(f(T))=\sum_{m \in M}\Prob(\xi=m) \log(\Prob(\xi=m)).
\end{align*}
Then by Theorem~\ref{thm:janson}, the probability that
\begin{align*}
F(T_n)=\log(\nu(T_n)) \leq (\mu+\varepsilon)n
\end{align*}
holds tends to $1$ for every $\varepsilon>0$ as $n \to \infty$. Thus, the statement follows.
\end{proof}

We are now able to derive the following theorem on the number of distinct plane trees represented by the fringe subtrees of a random tree of size $n$ from a simply generated family of trees:

\begin{theo}\label{thm:plane-simply-generated}
Let $J_n$ denote the number of distinct plane trees represented by the fringe subtrees of a random tree $T_n$ of size $n$ drawn from a simply generated family of trees $\mathcal{F}$ with weight sequence $(\phi_m)_{m \geq 0}$, and let $\Phi(x)=\sum_{m \geq 0}\phi_mx^m$. Let $R$ denote the radius of convergence of $\Phi$ and suppose that there exists $\tau \in (0,R]$ with $\tau\Phi'(\tau)=\Phi(\tau)$. Moreover, suppose that the offspring distribution $\xi$ of the Galton--Watson process corresponding to $\mathcal{F}$ satisfies $\mathbb{V}(\xi)<\infty$.
Set $\kappa = 2\tau^{-1}(\Phi(\tau))^{1/2}(2\pi\Phi''(\tau))^{-1/2}$. Furthermore, let $M =\{k \geq 0 \mid \phi_k>0\}$ and define the sequence $(\psi_k)_{k \geq 0}$ by $\psi_k=1$ if $k \in M$ and $\psi_k=0$ otherwise. Let $\Psi(x)=\sum_{k \geq 0}\psi_kx^k$, and let $\upsilon$ denote the solution to the equation $\upsilon\Psi'(\upsilon)=\Psi(\upsilon)$. Set 
\begin{align*}
C_1 =\log(\Psi'(\upsilon)) \quad \text{ and } \quad C_2 =-\mu,
\end{align*}
with $\mu$ defined as in Lemma~\ref{lem:plane-janson}. Then
\begin{itemize}
\item[(i)]$\displaystyle \kappa \sqrt{C_2} \frac{n}{\sqrt{\log n}}(1+o(1)) \leq \Ex(J_n) \leq \kappa \sqrt{ C_1}\frac{n}{\sqrt{\log n}}(1+o(1))$,
\item[(ii)]$\displaystyle \kappa \sqrt{ C_2} \frac{n}{\sqrt{\log n}}(1+o(1)) \leq J_n \leq \kappa \sqrt{C_1}\frac{n}{\sqrt{\log n}}(1+o(1))$ with high probability.
\end{itemize}
\end{theo}

\begin{proof}
Here we consider two trees as isomorphic if their plane representations are identical. This yields a partition of $\mathcal{F}_k$ into isomorphism classes $\mathcal{I}_k$, for which we will verify that the conditions of Theorem~\ref{thm:master-theorem-simply-generated} are satisfied. 
The number $|\mathcal{I}_k|$ of isomorphism classes equals the number of all plane trees of size $k$ with vertex degrees in $M$, 
which can be determined from Theorem~\ref{thm:number-of-trees}: the weight sequence $(\psi_k)_{k \geq 0}$ characterizes the simply generated family of plane trees with vertex degrees in $M$. We thus find by Theorem~\ref{thm:number-of-trees}:
\begin{align*}
\log(|\mathcal{I}_k|)=\log(\Psi'(\upsilon))k(1+o(1)),
\end{align*}
so condition (C1) is satisfied with
\begin{align*}
C_1 =\log(\Psi'(\upsilon)).
\end{align*}
Now we show that condition (C2) is satisfied as well. By Lemma~\ref{lem:plane-janson}, there exists a sequence of integers $k_j$  such that
$$\Prob\big( \nu(T_k)\leq e^{(\mu+1/j)k} \big) \geq 1 - \frac1j$$
for all $k \geq k_j$. So if we set $\varepsilon_k = \min \{ \frac1j \mid k_j \leq k\}$, then
$$\Prob\big(\nu(T_k)\leq e^{(\mu+\varepsilon_k)k} \big) \geq 1 - \varepsilon_k,$$
and $\varepsilon_k \to 0$ as $k \to \infty$. Now define the subset $\mathcal{J}_k \subseteq \mathcal{I}_k$ as the set of isomorphism classes of trees whose corresponding plane embedding $t$ satisfies $\nu(t) \leq e^{(\mu+\varepsilon_k)k}$. The probability that a random tree of size $k$ in $\mathcal{F}$ lies in an isomorphism class in the set $\mathcal{J}_k$ is precisely the probability that a conditioned Galton--Watson tree $T_k$ corresponding to the offspring distribution $\xi$ satisfies $\nu(T_k) \leq e^{(\mu+\varepsilon_k)k}$.
Thus we find that the probability that a random tree in $\mathcal{F}_k$ lies in an isomorphism class in the set  $\mathcal{J}_k$ tends to $1$ as $k \to \infty$.

Furthermore, the probability that a random tree $T_k$ of size $k$ in $\mathcal{F}$ has the shape of $t \in \mathcal{T}_k$ when regarded as a plane tree, i.e., the probability that $\mathcal{T}_k$ lies in the fixed isomorphism class $I \in \mathcal{J}_k$ containing all trees in the family $\mathcal{F}$ with plane representation $t$ is never greater than 
\begin{align*}
P_{\xi}(t)=\frac{\nu(t)}{\sum_{t' \in \mathcal{T}_k}\nu(t')}.
\end{align*}
In particular, the numerator is bounded by $e^{(\mu+\varepsilon_k)k}$ as $I \in \mathcal{J}_k$. In order to estimate the denominator, we apply Theorem~\ref{thm:number-of-trees}: we find that $\sum_{t' \in \mathcal{T}_n}\nu(t')$ is the total weight of all plane trees of size $n$ with respect to the weight sequence $(\Prob(\xi=k))_{k \geq 0}=(\phi_k\tau^k\Phi(\tau)^{-1})_{k \geq 0}$. If we set $\tilde{\Phi}(x) = \sum_{k \geq 0} \phi_k \tau^k \Phi(\tau)^{-1} x^k$, then $\tilde{\Phi}(1) = \tilde{\Phi}'(1) = 1$, and we obtain from Theorem~\ref{thm:number-of-trees} that
\begin{align}\label{eq:sum-nu}
\sum_{t \in \mathcal{T}_n}\nu(t)=\sqrt{\frac{\tilde{\Phi}(1)}{2\pi\tilde{\Phi}''(1)}}\frac{\tilde{\Phi}'(1)^n}{n^{3/2}} (1+O(n^{-1}))=\sqrt{\frac{\Phi(\tau)}{2\pi\tau^2\Phi''(\tau)}}n^{-3/2} (1+O(n^{-1})).
\end{align}
Hence,
\begin{align*}
P_{\xi}(t)\leq \sqrt{\frac{2\pi\tau^2\Phi''(\tau)}{\Phi(\tau)}}k^{3/2}e^{(\mu+\varepsilon_k)k}(1+O(k^{-1})) = e^{\mu k + o(k)},
\end{align*}
which shows that condition (C2) is satisfied with $C_2 = -\mu$. The statement of Theorem~\ref{thm:plane-simply-generated} follows from Theorem~\ref{thm:master-theorem-simply-generated}, as by \eqref{eq:xi-finite-variance}, we know that the variance of the Galton--Watson process corresponding to $\mathcal{F}$ is given by
\begin{align*}
\mathbb{V}(\xi)=\sigma^2=\frac{\tau^2\Phi''(\tau)}{\Phi(\tau)}.
\end{align*}
\end{proof}

We remark that for the family of plane trees, the statement of Theorem~\ref{thm:plane-simply-generated} is equivalent to the statement of Theorem~\ref{thm:ordered-simply-generated}: as $\phi_k=1$ for every $k \geq 0$ in this case, the constant $C_1$ in the upper bound of Theorem~\ref{thm:plane-simply-generated} evaluates to $\log(\Phi'(\tau))$. Furthermore, for every plane tree $t$ of size $n$, we have $\nu(t)/\sum_{t' \in \mathcal{T}_n} \nu(t')=1/y_n$, so that the constant $C_2$ in Theorem~\ref{thm:plane-simply-generated} evaluates to $\log(\Phi'(\tau))$ as well.

A plane representation of a binary tree is a Motzkin tree (see Example~\ref{ex:motzkin}). So for the family of binary trees, we obtain the following result regarding the number of distinct plane trees, i.e.,~Motzkin trees, represented by the fringe subtrees of a uniformly random binary tree of size $n$:
\begin{corollary}
Let $J_n$ denote the number of distinct plane trees represented by the fringe subtrees of a uniformly random binary tree of size $n$. Let
\begin{align*}
c_7 =\sqrt{\frac{6\log 2}{\pi}}\approx 1.1505709891 \text{ and }  c_8 =\frac{2\sqrt{\log 3}}{\sqrt{\pi}}\approx1.1827073223.
\end{align*}
Then
\begin{itemize}
\item[(i)] $\displaystyle c_7 \frac{n}{\sqrt{\log n}}(1+o(1))\leq \Ex(J_n)\leq c_8 \frac{n}{\sqrt{\log n}}(1+o(1))$,
\item[(ii)] $\displaystyle c_7 \frac{n}{\sqrt{\log n}}(1+o(1))\leq  J_n\leq c_8\frac{n}{\sqrt{\log n}}(1+o(1))$ with high probability.
\end{itemize}
\end{corollary}
\begin{proof}
The family of binary trees is obtained from the weight sequence $(\phi_k)_{k \geq 0}$ with $\Phi(x)=1+2x+x^2$. We find that $\Psi(x)=1+x+x^2$, with $\Psi$ defined as in Theorem~\ref{thm:plane-simply-generated}. Thus, $\upsilon=1$ solves the equation $\upsilon\Psi'(\upsilon)=\Psi(\upsilon)$ and $\Psi'(\upsilon)=3$. Hence, the constant $C_1$ in Theorem~\ref{thm:plane-simply-generated} evaluates to  $C_1=\log 3$. We remark again that the function $\Psi$ characterizes the family of Motzkin trees (Example~\ref{ex:motzkin}). The asymptotic growth of the number of Motzkin trees is well known, see e.g.~\cite{FlajoletS09}.
To compute the constant for the lower bound, we find that $\tau=1$ and $\Phi(\tau)=\Phi'(\tau)=4$. Hence, the offspring distribution $\xi$ of the Galton--Watson process corresponding to $\mathcal{F}$ is defined by
$\Prob(\xi=0)=1/4$, $\Prob(\xi=1)=1/2$ and $\Prob(\xi=2)=1/4$. We find
\begin{align*}
\mu=\sum_{k = 0}^2\Prob(\xi=k)\log(\Prob(\xi=k))=-\frac{3\log 2}{2},
\end{align*}
and hence $C_2=(3\log 2)/2$. With $\kappa = 2\tau^{-1}(\Phi(\tau))^{1/2}(2\pi\Phi''(\tau))^{-1/2}=2/\sqrt{\pi}$, the statement follows.
\end{proof}

Similarly, for the family of labelled trees, we obtain the following result (recall that we obtain a unique plane representation of a labelled tree if we first order the children of each vertex according to their labels and then disregard the vertex labels):

\begin{corollary}
Let $J_n$ denote the number of distinct plane trees represented by the fringe subtrees of a uniformly random labelled tree of size $n$. Let
\begin{align*}
c_9 =\bigg(\frac{2}{\pi} \Big(1 + \sum_{k \geq 2}\frac{\log(k!)}{e k!} \Big)\bigg)^{1/2}\approx 0.9114210724  \text{ and } c_{10}=\sqrt{\frac{2\log 4}{\pi}}\approx 0.9394372787.
\end{align*}
Then
\begin{itemize}
\item[(i)] $\displaystyle c_9 \frac{n}{\sqrt{\log n}}(1+o(1))\leq \Ex(J_n)\leq c_{10} \frac{n}{\sqrt{\log n}}(1+o(1))$,
\item[(ii)] $\displaystyle c_9 \frac{n}{\sqrt{\log n}}(1+o(1))\leq  J_n\leq c_{10}\frac{n}{\sqrt{\log n}}(1+o(1))$ with high probability.
\end{itemize}
\end{corollary}
\begin{proof}
The family of labelled trees is obtained as the simply generated family of trees with weight sequence $(\phi_k)_{k \geq 0}$ satisfying $\phi_k=1/k!$ for every $k \geq 0$. We find that $\Psi(x)=\sum_{k \geq 0}x^k$ and that $\upsilon=1/2$ solves the equation $\upsilon\Psi'(\upsilon)=\Psi(\upsilon)$, so that the constant $C_1$ in Theorem~\ref{thm:master-theorem-simply-generated} evaluates to $C_1=\log 4$. In order to compute the constant for the lower bound, we first notice that $\tau=1$ solves the equation $\tau\Phi'(\tau)=\Phi(\tau)$ with $\Phi(\tau)=e$. The offspring distribution $\xi$ of the Galton--Watson process corresponding to the family of labelled trees is well known to be the Poisson distribution (with $\Prob(\xi=k)=(ek!)^{-1}$ for every $k \geq 0$). Hence, we have
\begin{align*}
\mu=\sum_{k \geq 0}\Prob(\xi=k)\log(\Prob(\xi=k))=-e^{-1}\sum_{k \geq 0}\frac{1+\log\left(k!\right)}{k!}\approx -1.3048422423.
\end{align*}
With $\kappa = 2\tau^{-1}(\Phi(\tau))^{1/2}(2\pi\Phi''(\tau))^{-1/2}=\sqrt{2/\pi}$, the statement follows.
\end{proof}

\subsection{Distinct unordered fringe subtrees in simply generated trees}\label{subsec:unordered}

In this subsection, we apply Theorem~\ref{thm:master-theorem-simply-generated} to count the number of distinct unordered trees represented by the fringe subtrees of a random tree of size $n$ drawn randomly from a simply generated family of trees. 
Thus we consider two trees from the family $\mathcal{F}$ as isomorphic if their unordered representations are identical.
This is meaningful for all simply generated families, since every rooted tree has a natural unordered representation. 
Let $t \in \mathcal{T}$ be a plane tree. As a simple application of the orbit-stabilizer theorem, one finds that the number of plane trees with the same unordered representation as $t$ is given by
\begin{align*}
\frac{\prod_{v \in t}\deg(v)!}{|\Aut(t)|},
\end{align*}
where $|\Aut(t)|$ denotes the cardinality of the automorphism group of $t$. This is because the permutations of the branches at the different vertices of $t$ generate a group of order $\prod_{v \in t}\deg(v)!$ acting on the plane trees with the same unordered representations as $t$, and $|\Aut t|$ is the subgroup that fixes $t$. It follows that
$$\nu(t) \frac{\prod_{v \in t}\deg(v)!}{|\Aut(t)|}$$
is the total weight of all plane representations of $t$ within a simply generated family. This quantity will play the same role that $\nu(t)$ played in the proof of Theorem~\ref{thm:plane-simply-generated}. From Theorem~\ref{thm:janson}, we obtain the following result:

\begin{lemma}\label{lem:unordered-janson}
Let $\xi$ be the offspring distribution of a critical Galton--Watson process satisfying $\mathbb{V}(\xi)=\sigma^2<\infty$, and let $T_k$ be a conditioned Galton--Watson tree of size $k$ with respect to $\xi$. 
Then there is a constant $\lambda < 0$ such that the probability that
\begin{align*}
\nu(T_k)\frac{\prod_{v \in T_k}\deg(v)!}{|\Aut (T_k)|}\leq e^{(\lambda+\varepsilon)k}
\end{align*}
holds tends to $1$ for every $\varepsilon>0$ as $k \to \infty$. 
\end{lemma}
\begin{proof}
As in the proof of Lemma~\ref{lem:plane-janson}, we aim to define a suitable additive functional. To this end, we need a recursive description of $|\Aut t|$, the size of the automorphism group. Let $\rho(t)$ again denote the degree of the root vertex of $t$, let $t_1,t_2,\ldots,t_{\rho(t)}$ be the root branches of a tree $t$, and let $m_1, m_2, \ldots, m_{k_{t}}$ denote the multiplicities of isomorphic branches of $t$ ($m_1+m_2+\cdots+m_{k_t} = \rho(t)$). Here we call two trees isomorphic if they are identical as unordered trees. That is, the $\rho(t)$ many subtrees rooted at the children of the root vertex fall into $k_{t}$ many different isomorphism classes, where $m_i$ of them belong to isomorphism class $i$, respectively.
Then we have
$$|\Aut(t)| = \prod_{j=1}^{\rho(t)} |\Aut(t_j)| \cdot \prod_{i=1}^{k_t} m_i!\,,$$
since an automorphism of $t$ acts as an automorphism within branches and also possibly permutes branches that are isomorphic. In fact, the whole structure of $\Aut(t)$ is well understood \cite{jordan1869}.
It follows from the recursion for $|\Aut(t)|$ that 
\begin{align*}
F(t)= \log\left(\frac{\nu(t)\prod_{v \in t}\deg(v)!}{|\Aut(t)|}\right)
\end{align*}
(well-defined for all $t$ with $\nu(t) > 0$) is the additive functional associated with the toll function $f$ that is defined by
\begin{equation}\label{eq:toll_unordered}
f(t)=\begin{cases}\log(\Prob(\xi=\rho(t))\rho(t)!)-\log(m_1!m_2! \cdots m_{k_t}!) \quad &\text{if }\Prob(\xi=\rho(t))>0,\\
0 &\text{otherwise.}
\end{cases}
\end{equation}
Let $M =\{m \geq 0 \mid \Prob(\xi=m)>0\}$, and 
let $T$ be the unconditioned Galton--Watson tree corresponding to $\xi$. Since
$$0 \leq \log(\rho(t)!)-\log(m_1!m_2! \cdots m_{k_t}!) \leq \log(\rho(t)!),$$
we have
\begin{align*}
\Ex(|f(T)|) \leq \sum_{m \in M}\Prob(\xi=m)|\log(\Prob(\xi=m))|+\sum_{m \in M}\Prob(\xi=m)|\log(m!)|.
\end{align*}
The first sum was shown to be finite earlier in \eqref{eq:expectation-f-finite}, and the second sum is finite as $\log(m!) = O(m^2)$ and $\mathbb{V}(\xi)<\infty$ by assumption. Moreover, we find
\begin{align*}
|\Ex(f(T_k))|\leq \sum_{m \in M \atop m\leq k}\Prob(\rho(T_k)=m)|\log(\Prob(\xi=m)m!)|.
\end{align*}
Again by result (2.7) in \cite{janson05}, there is a constant $c>0$ (independent of $k$ and $m$) such that
\begin{align*}
\Prob(\rho(T_k)=m)\leq cm\Prob(\xi=m)
\end{align*}
for all $m,k\geq 0$. We thus find 
\begin{align*}
|\Ex(f(T_k))|&\leq c\sum_{m \in M \atop m\leq k}m\Prob(\xi=m)|\log\left(\Prob(\xi=m)m!\right)|\\
&\leq c\sum_{m \in M}m\Prob(\xi=m)|\log\left(\Prob(\xi=m)\right)|+c\sum_{m \in M \atop m\leq k}m\Prob(\xi=m)\log\left(m!\right).
\end{align*}
The first sum was shown to be finite in \eqref{eq:toll-function-finite}. As $\log(m!)\leq m\log m$, we obtain for the second sum:
\begin{align*}
\sum_{m \in M \atop m\leq k}m\Prob(\xi=m)\log\left(m!\right)\leq \log k \sum_{m \in M}m^2\Prob(\xi=m) = O(\log k),
\end{align*}
as by assumption, $\Ex(\xi)=1$ and $\mathbb{V}(\xi)<\infty$.
In particular, we thus have $\Ex|f(T_k)|= O(\log k)$. Altogether, we find that the requirements of Theorem~\ref{thm:janson} are satisfied. Now set
\begin{align*}
\lambda=\Ex(f(T)).
\end{align*}
By Theorem~\ref{thm:janson}, the probability that
\begin{align*}
F(T_k)=\log\left(\frac{\nu(T_k)\prod_{v \in T_k}\deg(v)!}{|\Aut(T_k)|}\right)\leq (\lambda+\varepsilon)k
\end{align*}
holds tends to $1$ for every $\varepsilon >0$ as $k \to \infty$. Thus, the statement follows. 
\end{proof}

Additionally, we need the following result on the number of unordered trees with vertex degrees from some given set $M \subseteq \mathbb{N}$:
\begin{theo}[\hspace{1sp}{{\cite[pp.~71-72]{FlajoletS09}}}]\label{thm:growth-unordered}
Let $M \subseteq \mathbb{N}$ with $0 \in M$, and let $u_k^M$ denote the number of unordered rooted trees $t$ of size $k$ with the property that the outdegree of every vertex in $t$ lies in $M$. Then
\begin{align*}
u_k^M \sim a_M \cdot \frac{b_M^k}{k^{3/2}}
\end{align*}
if $k \equiv 1 \mod d$, where $d$ is the greatest common divisor of all elements of $M$, and $u_k=0$ otherwise, where the constants $a_M,b_M$ depend on $M$. 
\end{theo}
Again for the sake of simplicity, we assume that $d=1$ holds for all families of trees considered in the following.
We are now able to derive a theorem on the number of distinct unordered trees represented by the fringe subtrees of a random tree of size $n$ drawn from a simply generated family of trees.

\begin{theo}\label{thm:simply-generated-unordered}
Let $K_n$ denote the total number of distinct unordered trees represented by the fringe subtrees of a random tree $T_n$ of size $n$ drawn from a simply generated family of trees $\mathcal{F}$ with weight sequence $(\phi_k)_{k \geq 0}$, and let $\Phi(x)=\sum_{m \geq 0}\phi_mx^m$. Let $R$ denote the radius of convergence of $\Phi$ and suppose that there exists $\tau \in (0,R]$ with $\tau\Phi'(\tau)=\Phi(\tau)$. Moreover, suppose that the offspring distribution $\xi$ of the Galton--Watson process corresponding to $\mathcal{F}$ satisfies $\mathbb{V}(\xi)=\sigma^2<\infty$. Set $\kappa=2\tau^{-1}(\Phi(\tau))^{1/2}(2\pi\Phi''(\tau))^{-1/2}$. Furthermore, let $M =\{m \in \mathbb{N} \mid \phi_m>0\}$ and set $C_1 =\log(b_M)$, where $b_M$ is the constant in Theorem~\ref{thm:growth-unordered}, and $C_2 =-\lambda$, where $\lambda$ is the constant in Lemma~\ref{lem:unordered-janson}. Then
\begin{itemize}
\item[(i)]$\displaystyle \kappa \sqrt{C_2} \frac{n}{\sqrt{\log n}}(1+o(1)) \leq \Ex(K_n) \leq \kappa \sqrt{C_1} \frac{n}{\sqrt{\log n}}(1+o(1)),$
\item[(ii)]$\displaystyle \kappa \sqrt{C_2} \frac{n}{\sqrt{\log n}}(1+o(1)) \leq K_n \leq \kappa \sqrt{C_1} \frac{n}{\sqrt{\log n}}(1+o(1))$ with high probability.
\end{itemize}
\end{theo}
\begin{proof}
Here we consider two trees as isomorphic if their unordered representations are identical. This yields a partition of $\mathcal{F}_k$ into isomorphism classes $\mathcal{I}_k$, for which we will verify that the conditions of Theorem~\ref{thm:master-theorem-simply-generated} are satisfied. 
The number $|\mathcal{I}_k|$ of isomorphism classes equals the number of all unordered trees of size $k$ with vertex degrees in $M$, which is given by Theorem~\ref{thm:growth-unordered}: we have
\begin{align*}
\log(|\mathcal{I}_k|)=\log(b_M)k(1+o(1)).
\end{align*}
Hence, condition (C1) is satisfied with $C_1=\log(b_M)$.
Note that if two plane trees $t, t' \in \mathcal{T}$ have the same unordered representation, we have $\nu(t)=\nu(t')$, $\prod_{v \in t}\deg(v)!=\prod_{v \in t'}\deg(v)!$ and $|\Aut(t)|=|\Aut(t')|$ (it is thus well-defined to define $\nu(u)=\nu(t)$ for a plane tree $t \in \mathcal{T}$ and its unordered representation $u$). As in the proof of Theorem~\ref{thm:plane-simply-generated}, we can now use Lemma~\ref{lem:unordered-janson} to show that there exists a sequence $\varepsilon_k$ that tends to $0$ as $k \to \infty$ with the property that
$$\Prob\Big(\nu(T_k) \frac{\prod_{v \in T_k}\deg(v)!}{|\Aut(T_k)|} \leq e^{(\lambda+\varepsilon_k)k} \Big) \geq 1 - \varepsilon_k.$$
So let $\mathcal{J}_k\subseteq \mathcal{I}_k$ denote the subset of isomorphism classes of trees in $\mathcal{F}_k$ such that the trees $t$ that they represent satisfy
\begin{align*}
\nu(t)\frac{\prod_{v \in t}\deg(v)!}{|\Aut(t)|}\leq e^{(\lambda+\varepsilon_k)k}.
\end{align*}
The probability that a random tree of size $k$ drawn from $\mathcal{F}_k$ lies in an isomorphism class that belongs to the set $\mathcal{J}_k$ is precisely the probability that a conditioned Galton--Watson tree $T_k$ of size $k$ with offspring distribution $\xi$ satisfies
\begin{align*}
\nu(T_k)\frac{\prod_{v \in T_k}\deg(v)!}{|\Aut(T_k)|}\leq e^{(\lambda+\varepsilon_k)k},
\end{align*}
which is at least $1 - \varepsilon_k$ by construction. Thus condition (C2a) is satisfied. 

Now let $I \in \mathcal{J}_k$ be a single isomorphism class, and let $u$ be the unordered tree that it represents.
The probability that a random tree in $\mathcal{F}$ of size $k$ lies in the isomorphism class $I$ is
\begin{align*}
\frac{\nu(u)}{\sum_{t\in \mathcal{T}_k}\nu(t)}\frac{\prod_{v \in u}\deg(v)!}{|\Aut(u)|},
\end{align*}
since $\prod_{v \in t}\deg(v)!/|\Aut(t)|$ equals the number of plane representations of the tree $u$, each of which has probability $\nu(u)$.
As explained in the proof of Theorem~\ref{thm:plane-simply-generated} (see \eqref{eq:sum-nu}), we have
\begin{align*}
\sum_{t \in \mathcal{T}_k}\nu(t)=\sqrt{\frac{\Phi(t)}{2\pi\tau^2\Phi''(\tau)}}k^{-3/2}(1+O(k^{-1})).
\end{align*}
Thus, the probability that a random tree in $\mathcal{F}$ of size $k$ lies in a single isomorphism class $I \in \mathcal{J}_k$ is never greater than
\begin{align*}
\sqrt{\frac{2\pi\tau^2\Phi''(\tau)}{\Phi(\tau)}}k^{3/2}e^{(\lambda+\varepsilon_k)k}(1+O(k^{-1})) = e^{\lambda k + o(k)}.
\end{align*}
So condition (C2b) is satisfied as well, with $C_2 = -\lambda$. Theorem~\ref{thm:simply-generated-unordered} now follows directly from Theorem~\ref{thm:master-theorem-simply-generated}.
\end{proof}

In order to obtain bounds on the number $K_n$ of distinct unordered trees represented by the fringe subtrees of a random tree $T_n$ drawn from some concrete family of trees, we need to determine the values of the constants  $\lambda$ and $b_M$ in Lemma~\ref{lem:unordered-janson} and Theorem~\ref{thm:growth-unordered} for the particular family of trees. For the family of binary trees, these values follow from known results. 
The number of unordered rooted trees of size $k$ with vertex degrees in $M=\{0,1,2\}$ is given by the $(k+1)$st \emph{Wedderburn-Etherington number} $W_{k+1}$. The asymptotic growth of these numbers is
\begin{align*}
W_k \sim a_M \cdot k^{-3/2}\cdot b_M^k,
\end{align*}
for certain positive constants $a_M, b_M$ \cite{bonaf09, finch03}. In particular, we have $b_M \approx 2.4832535363$. 

In order to determine a concrete value for the constant $\lambda$ in Lemma~\ref{lem:unordered-janson} for the family of binary trees, we make use of a theorem by B\'{o}na and Flajolet \cite{bonaf09} on the number of automorphisms of a uniformly random full binary tree: a \emph{full binary tree} is a binary tree where each vertex has either exactly two or zero descendants, i.e., there are no unary vertices. Note that every full binary tree with $2k-1$ vertices consists of $k$ leaves and $k-1$ binary vertices, thus it is often convenient to define the size of a full binary tree as its number of leaves. 
The following theorem is stated for phylogenetic trees in \cite{bonaf09}, but the two probabilistic models are equivalent:

\begin{theo}[{see \cite[Theorem 2]{bonaf09}}]\label{thm:bona-flajolet}
Consider a uniformly random full binary tree $T_k$ with $k$ leaves, and let $|\Aut(T_k)|$ be the cardinality of its automorphism group. The logarithm of this random variable satisfies a central limit theorem: For certain positive constants $\gamma$ and $\beta$, we have
\begin{align*}
\Prob(|\Aut(T_k)| \leq 2^{\gamma k + \beta \sqrt{k} x}) \overset{k \to \infty}{\to} \frac{1}{\sqrt{2\pi}} \int_{-\infty}^x e^{-t^2/2}\,dt
\end{align*}
for every real number $x$. The numerical value of the constant $\gamma$ is $0.2710416936$. 
\end{theo}

The simply generated family of full binary trees corresponds to the weight sequence with $\phi_0=\phi_2=1$ and $\phi_j=0$ for $j \notin \{0,2\}$. The corresponding offspring distribution $\xi_1$ satisfies $\Prob(\xi_1=0)=\Prob(\xi_1=2)=1/2$. Let $t$ denote a (plane representation of) a full binary tree of size $n = 2k-1$, with $k$ leaves and $k-1$ internal vertices. Then $\nu_{\xi_1}(t)=2^{-2k+1}$ and $\prod_{v \in t}\deg(v)!=2^{k-1}$, and consequently
$$\nu_{\xi_1}(T_k)\frac{\prod_{v \in T_k}\deg(v)!}{|\Aut (T_k)|} = \frac{1}{2^k |\Aut(T_k)|}$$
for a random full binary tree $T_k$ with $k$ leaves. It follows from Theorem~\ref{thm:bona-flajolet} that
$$\frac{1}{2k-1} \log \Big( \nu_{\xi_1}(T_k)\frac{\prod_{v \in T_k}\deg(v)!}{|\Aut (T_k)|} \Big) \overset{p}{\to} - \frac{(1+\gamma) \log 2}{2},$$
thus $\lambda = - \frac{(1+\gamma)\log 2}{2} \approx -0.4405094831$ in this special case. As the numbers of unordered rooted trees with vertex degrees in $M=\{0,2\}$ are counted by the \emph{Wedderburn-Etherington numbers} as well \cite{bonaf09}, we obtain the following corollary from Theorem \ref{thm:simply-generated-unordered}:

\begin{corollary}
Let $K_n$ denote the number of distinct unordered trees represented by the fringe subtrees of a uniformly random full binary tree with $n$ leaves. Then for $c_1 \approx 1.0591261434$ and $c_2 \approx 1.0761505454$, we have 
\begin{itemize}
\item[(i)]$\displaystyle c_1 \frac{n}{\sqrt{\log n}}(1+o(1))\leq \Ex(K_n) \leq c_2 \frac{n}{\sqrt{\log n}}(1+o(1))$,
\item[(ii)]$\displaystyle c_1 \frac{n}{\sqrt{\log n}}(1+o(1))\leq K_n \leq c_2\frac{n}{\sqrt{\log n}}(1+o(1))$ with high probability.
\end{itemize}
\end{corollary}

In order to obtain a corresponding result for binary trees rather than full binary trees, we observe that as every full binary tree with $k$ leaves has exactly $k-1$ internal vertices, there is a natural one-to-one correspondence between the set of full binary trees with $k$ leaves and the set of binary trees with $k-1$ vertices. Let $\vartheta(t)$ denote the binary tree of size $k-1$ obtained from a full binary tree $t$ with $k$ leaves by removing the leaves of $t$ and only keeping the internal vertices of $t$. Then $\vartheta$ is a bijection between the set of full binary trees with $k$ leaves and the set of binary trees of size $k-1$ for every $k \geq 2$. Fringe subtrees of $t$ correspond to fringe subtrees of $\vartheta(t)$ and vice verca, except for the leaves of $t$. Thus $t$ and $\vartheta(t)$ have almost the same number of non-isomorphic fringe subtrees (the difference is exactly $1$). If $T_k$ is a uniformly random full binary tree with $k$ leaves, then $\vartheta(T_k)$ is a uniformly random binary tree of size $k-1$. Hence, in view of this correspondence between binary trees and full binary trees, Theorem~\ref{thm:unordered-binary} follows.

As another example, we take the family of labelled trees. Here, we have $M = \{0,1,2,\ldots\}$, and the number of isomorphism classes is the number of \emph{P\'olya trees} (rooted unordered trees), which follows the same kind of asymptotic formula as the Wedderburn-Etherington numbers above, with a growth constant $b_M \approx 2.9557652857$, see \cite{Otter1948}, \cite[Section VII.5]{FlajoletS09} or \cite[Section 5.6]{finch03}. This gives us immediately the value of $C_1 = \log(b_M)$.

The number of automorphisms satisfies a similar central limit theorem as in Theorem~\ref{thm:bona-flajolet}, with a constant $\gamma \approx 0.0522901096$ (and $k$ being the number of vertices rather than the number of leaves), see \cite{BIRStalk}. Since the expression $\log(\Prob(\xi=\rho(t))\rho(t)!)$ in~\eqref{eq:toll_unordered} simplifies to $-1$ for every value of $\rho(t)$ in the case of labelled trees, we have $\lambda = - 1 - \gamma$ and thus $C_2 = 1 + \gamma \approx 1.0522901096$. Finally, $\kappa = \sqrt{2/\pi}$ in this example. Putting everything together, we obtain

\begin{corollary}\label{cor:distinct_unordered_labelled}
Let $K_n$ denote the number of distinct unordered trees represented by the fringe subtrees of a uniformly random labelled tree with $n$ vertices. Then for $c_{11} \approx 0.8184794989$ and $c_{12} \approx 0.8306271816$, we have 
\begin{itemize}
\item[(i)]$\displaystyle c_{11} \frac{n}{\sqrt{\log n}}(1+o(1))\leq \Ex(K_n) \leq c_{12} \frac{n}{\sqrt{\log n}}(1+o(1))$,
\item[(ii)]$\displaystyle c_{11} \frac{n}{\sqrt{\log n}}(1+o(1))\leq K_n \leq c_{12}\frac{n}{\sqrt{\log n}}(1+o(1))$ with high probability.
\end{itemize}
\end{corollary}

\section{Applications: increasing trees}

We now prove the analogues of the previous section for increasing trees by verifying that the conditions of Theorem~\ref{thm:master-theorem-increasing} are satisfied. Note that (C1) still holds in all cases for the same reasons as before. Only condition (C2) requires some effort.

Once again, we make use of results on additive functionals. For additive functionals of increasing trees with finite support, i.e., for functionals for which there exists a constant $K$ such that $f(t)=0$ whenever $|t|>K$, a central limit was proven in~\cite{HolmgrenJS17} and~\cite{RalaivaosaonaW19} (the latter even contains a slightly more general result). Those results do not directly apply to the additive functionals that we are considering here. However, convergence in probability is sufficient for our purposes. We have the following lemma:

\begin{lemma}\label{lem:additive_functionals_increasing}
Let $T_n$ denote a random tree with $n$ vertices from one of the very simple families of increasing trees (recursive trees, $d$-ary increasing trees, gports), and let $F$ be any additive functional with toll function $f$. As before, set $\alpha = 0$ for recursive trees, $\alpha = - \frac1{d}$ for $d$-ary increasing trees and $\alpha = \frac{1}{r}$ for gports. We have
$$\Ex(F(T_n)) = \Ex(f(T_n)) + \sum_{k = 1}^{n-1} \frac{((1+\alpha)n-\alpha)\Ex(f(T_k))}{((1+\alpha)k+1)((1+\alpha)k-\alpha)}.$$
Moreover, if $\Ex|f(T_n)| = o(n)$ and $\sum_{k=1}^{\infty} \frac{\Ex|f(T_k)|}{k^2} < \infty$, then we have
$$\frac{F(T_n)}{n} \overset{p}{\to} \mu = \sum_{k = 1}^{\infty} \frac{(1+\alpha) \Ex(f(T_k))}{((1+\alpha)k+1)((1+\alpha)k-\alpha)}.$$
\end{lemma}

\begin{proof}
The first statement follows directly from Lemma~\ref{lemma:increasing-expectation-variance}, since fringe subtrees are, conditioned on their size, again random trees following the same probabilistic model as the whole tree. For functionals with finite support, where $f(T) = 0$ for all but finitely many trees $T$, convergence in probability follows from the central limit theorems in~\cite{HolmgrenJS17} and~\cite{RalaivaosaonaW19}. For the more general case, we approximate the additive functional $F$ with a truncated version $F_m$ based on the toll function
$$f_m(T) = \begin{cases} f(T) & |T| \leq m, \\ 0 & \text{ otherwise.}\end{cases}$$
Since we already know that convergence in probability holds for functionals with finite support, we have
$$\frac{F_m(T_n)}{n} \overset{p}{\to} \mu_m = \sum_{k = 1}^{m} \frac{(1+\alpha) \Ex(f(T_k))}{((1+\alpha)k+1)((1+\alpha)k-\alpha)}.$$
Now we use the triangle inequality and Markov's inequality to estimate $\Prob(|F(T_n)/n - \mu| > \varepsilon)$. Choose $m$ sufficiently large so that $|\mu_m - \mu| < \frac{\varepsilon}{3}$. Then we have, for $n > m$,
\begin{align*}
\Prob\Big( \Big| \frac{F(T_n)}{n} - \mu \Big| > \varepsilon \Big) &\leq \Prob\Big( \Big| \frac{F_m(T_n)}{n} - \mu_m \Big| > \frac{\varepsilon}{3} \Big) + \Prob\Big( \Big| \frac{F_m(T_n) - F(T_n)}{n} \Big| > \frac{\varepsilon}{3} \Big)  \\
&\leq \Prob\Big( \Big| \frac{F_m(T_n)}{n} - \mu_m \Big| > \frac{\varepsilon}{3} \Big) + \frac{3}{\varepsilon} \Ex \Big| \frac{F_m(T_n) - F(T_n)}{n} \Big| \\
&\leq \Prob\Big( \Big| \frac{F_m(T_n)}{n} - \mu_m \Big| > \frac{\varepsilon}{3} \Big) \\
&\quad + \frac{3}{\varepsilon} \Big( \frac{\Ex|f(T_n)|}{n} + \sum_{k = m+1}^{n-1} \frac{((1+\alpha)n-\alpha)\Ex|f(T_k)|}{n((1+\alpha)k+1)((1+\alpha)k-\alpha)} \Big).
\end{align*}
Since $\frac{F_m(T_n)}{n} \overset{p}{\to} \mu_m$, it follows that
$$\limsup_{n \to \infty} \Prob\Big( \Big| \frac{F(T_n)}{n} - \mu \Big| > \varepsilon \Big) \leq \frac{3}{\varepsilon} \sum_{k = m+1}^{\infty} \frac{(1+\alpha)\Ex|f(T_k)|}{((1+\alpha)k+1)((1+\alpha)k-\alpha)}.$$
Taking $m \to \infty$, we finally find that
$$\lim_{n \to \infty} \Prob\Big( \Big| \frac{F(T_n)}{n} - \mu \Big| > \varepsilon \Big) = 0,$$
completing the proof.
\end{proof}

In order to apply this lemma in the same way as for simply generated trees, we need one more ingredient: let $t$ be a plane tree with $n$ vertices. The number of increasing labellings of the vertices with labels $1,2,\ldots,n$ is given by
$$\frac{n!}{\prod_{v} |t(v)|},$$
see for example \cite[Eq.~(5)]{Ruskey1992Generating} or \cite[Section 5.1.4, Exercise 20]{Knuth1998}. Considering a tree as a poset, this is equivalent to counting linear extensions. The quantity
$$\sum_{v} \log |t(v)|,$$
i.e., the sum of the logarithms of all fringe subtree sizes, is also known as the \emph{shape functional}, see \cite{FillK2004}.

\subsection{Distinct fringe subtrees and distinct plane fringe subtrees in increasing trees}

In this section, we consider increasing trees with a plane embedding. There is a natural embedding for $d$-ary increasing trees, where each vertex has $d$ possible positions at which a child can be attached. Similarly, plane oriented recursive trees can be regarded as plane trees with increasing vertex labels. In these cases, the notion of distinctness as in Section~\ref{subsec:distinct} is still meaningful: two fringe subtrees are considered the same if they have the same shape (as $d$-ary tree/plane tree) when the labels are removed.

\medskip

Let us start with $d$-ary increasing trees. In this case, the isomorphism classes are precisely $d$-ary trees (Example~\ref{ex:dary}), whose number is
$$|\mathcal{I}_k| = \frac{1}{k} \binom{dk}{k-1}.$$
It follows that
$$C_1 = \limsup_{k \to \infty} \frac{\log |\mathcal{I}_k|}{k} = d \log d - (d-1)\log(d-1),$$
so (C1) is satisfied. See also the discussion in the proof of Corollary~\ref{cor:distinct-binary}. 

\medskip

We now verify (C2). Taking the number of increasing labellings into account, as explained above, we find that for a given $d$-ary tree $t$ with $n$ vertices, the probability that a random increasing $d$-ary tree with $n$ vertices has the shape of $t$ is
$$\frac{n!}{\prod_{k=1}^{n-1} (1+k(d-1))} \prod_{v} \frac{1}{|t(v)|}.$$
Recall here that the denominator $\prod_{k=1}^{n-1} (1+k(d-1))$ is precisely the number of $d$-ary increasing trees with $n$ vertices. Note next that
$$\frac{n!}{\prod_{k=1}^{n-1} (1+k(d-1))} \sim \frac{\Gamma(\frac{1}{d-1}) n^{(d-2)/(d-1)}}{(d-1)^n} = \exp \big({- \log(d-1) n + o(\log n) }\big).$$
The additive functional with toll function $f(t) = \log |t|$ clearly satisfies the conditions of Lemma~\ref{lem:additive_functionals_increasing}, with
$$\mu = \sum_{k = 1}^{\infty} \frac{(1-1/d) \log k}{((1-1/d)k+1)((1-1/d)k+1/d)} = d(d-1) \sum_{k = 1}^{\infty} \frac{\log k}{((d-1)k+d)((d-1)k+1)}.$$
Thus it is possible (as in the proofs of Theorem~\ref{thm:plane-simply-generated} and Theorem~\ref{thm:simply-generated-unordered}) to define subsets $\mathcal{J}_k \subseteq \mathcal{I}_k$ of $d$-ary increasing trees with the property that the shape of a random $d$-ary increasing tree with $k$ vertices belongs to $\mathcal{J}_k$ with probability $1-o(1)$ as the number of vertices goes to infinity, while the probability of any single isomorphism class in $\mathcal{J}_k$ is never greater than $e^{-(\log(d-1) + \mu) k + o(k)}$. So condition (C2) is also satisfied, with a constant
$$C_2 = \log(d-1) + d(d-1) \sum_{k = 1}^{\infty} \frac{\log k}{((d-1)k+d)((d-1)k+1)}.$$
Hence we obtain the following theorem as a corollary of Theorem~\ref{thm:master-theorem-increasing}.

\begin{theo}
Let $H_n$ be the number of distinct $d$-ary trees occurring among the fringe subtrees of a random $d$-ary increasing tree of size $n$. For the two constants 
$$\underline{c}(d) = \frac{d}{d-1} \log(d-1) + d^2 \sum_{k = 1}^{\infty} \frac{\log k}{((d-1)k+d)((d-1)k+1)},$$
$$\overline{c}(d) = \frac{d}{d-1} \big( d \log d - (d-1)\log(d-1) \big)$$
the following holds:
\begin{enumerate}
\item[(i)] $\displaystyle \frac{\underline{c}(d) n}{\log n} (1+o(1)) \leq \Ex(H_n) \leq \frac{\overline{c}(d) n}{\log n} (1+o(1))$,
\item[(ii)] $\displaystyle \frac{\underline{c}(d) n}{\log n} (1+o(1)) \leq H_n \leq \frac{\overline{c}(d) n}{\log n} (1+o(1))$ with high probability.
\end{enumerate}
\end{theo}

In the special case $d=2$, which corresponds to binary search trees, we have $\underline{c}(2) \approx 2.4071298335$ and $\overline{c}(2) \approx 2.7725887222$, cf.~Theorem~\ref{thm:orderedbst}. This was already obtained in the conference version of this paper, see \cite{SeelbachWagner20}.

\medskip

For plane oriented recursive trees, the procedure is analogous. The isomorphism classes are precisely the plane trees (see Example~\ref{ex:plane}), and we have
$$|\mathcal{I}_k| = \frac{1}{k} \binom{2k-2}{k-1},$$
thus $C_1 = \log 4$. Moreover, arguing in the same way as for $d$-ary trees, we find that (C2) is satisfied with
$$C_2 = \log 2 + \sum_{k=1}^{\infty} \frac{2 \log k}{(2k+1)(2k-1)}.$$
So Theorem~\ref{thm:master-theorem-increasing} yields

\begin{theo}\label{theo:port}
Let $H_n$ be the number of distinct fringe subtrees in a random plane oriented recursive tree of size $n$. For the two constants 
$$c_{13} = \frac{\log 2}{2} + \sum_{k = 1}^{\infty} \frac{\log k}{(2k+1)(2k-1)} \approx 0.5854804841$$
and $c_{14} = \log 2 \approx 0.6931471806$,
the following holds:
\begin{enumerate}
\item[(i)] $\displaystyle \frac{c_{13} n}{\log n} (1+o(1)) \leq \Ex(H_n) \leq \frac{c_{14} n}{\log n} (1+o(1))$,
\item[(ii)] $\displaystyle \frac{c_{13} n}{\log n} (1+o(1)) \leq H_n \leq \frac{c_{14} n}{\log n} (1+o(1))$ with high probability.
\end{enumerate}
\end{theo}

For $d$-ary increasing trees, the notion of distinctness of Section~\ref{subsec:distinctplane} also makes sense (for plane oriented recursive trees, it is simply equivalent to that of Theorem~\ref{theo:port}). In this case, we consider fringe subtrees as distinct only if they are different as plane trees. Thus the isomorphism classes are plane trees with maximum degree at most $d$, which form a simply generated family of trees. Their generating function $Y_d(x)$ satisfies
$$Y_d(x) = x(1+Y_d(x) + Y_d(x)^2 + \cdots + Y_d(x)^d).$$
Letting $\tau_d$ be the unique positive solution of the equation $1 = t^2+2t^3 + \cdots + (d-1)t^d$, the exponential growth constant of this simply generating family is $\eta_d = \frac{1+\tau_d + \tau_d^2 + \cdots + \tau_d^d}{\tau_d} = 1 + 2\tau_d + 3\tau_d^2 + \cdots + d\tau_d^{d-1}$, see Theorem~\ref{thm:number-of-trees}. Thus (C1) is satisfied with $C_1 = \log \eta_d$. We also note that $\eta_d \in [3,4]$. Specifically, in the special case $d=2$ we obtain the Motzkin numbers with $\eta_2 = 3$, see Example~\ref{ex:motzkin}. Moreover, we have $\lim_{d \to \infty} \eta_d = 4$.

\medskip

In order to verify (C2) and determine a suitable constant, we combine the argument from the previous two theorems with that of Section~\ref{subsec:distinctplane}. The probability that a random increasing $d$-ary tree with $n$ vertices has the shape of $t$, regarded as a plane tree, is
$$\frac{n!}{\prod_{k=1}^{n-1} (1+k(d-1))} \prod_{v} \frac{\binom{d}{\deg(v)}}{|t(v)|}.$$

Note here that the product $\prod_v \binom{d}{\deg(v)}$ gives the number of $d$-ary realizations of the plane tree~$t$, see the proof of Theorem~\ref{thm:plane-simply-generated} for comparison. So we consider the additive functional with toll function $f(t) =\log |t| - \log \binom{d}{\rho(t)}$, where $\rho(t)$ is the degree of the root of $t$, instead of just $f(t) = \log |t|$ as it was chosen before. Since $\binom{d}{\rho(t)}$ is clearly bounded, the conditions of Lemma~\ref{lem:additive_functionals_increasing} are still satisfied, and we obtain a suitable constant $C_2$ that satisfies (C2) as before. For example, in the binary case we have the following theorem:

\begin{theo}
Let $J_n$ be the number of distinct plane trees occurring among the fringe subtrees of a random binary increasing tree of size $n$. For the two constants
$$c_{15} = 4 \sum_{k = 2}^{\infty} \frac{\log k - \frac{2 \log 2}{k}}{(k+1)(k+2)}=4 \sum_{k = 2}^{\infty} \frac{\log k }{(k+1)(k+2)}- \frac{2\log 2 }{3} \approx 1.9450317130$$
and $c_{16} = 2 \log 3 \approx 2.1972245773$, the following holds:
\begin{enumerate}
\item[(i)] $\displaystyle \frac{c_{15} n}{\log n} (1+o(1)) \leq \Ex(J_n) \leq \frac{c_{16} n}{\log n} (1+o(1))$,
\item[(ii)] $\displaystyle \frac{c_{15} n}{\log n} (1+o(1)) \leq J_n \leq \frac{c_{16} n}{\log n} (1+o(1))$ with high probability.
\end{enumerate}
\end{theo}

\subsection{Distinct unordered fringe subtrees in increasing trees}

The notion of distinctness of Section~\ref{subsec:unordered} is meaningful for all families of increasing trees we are considering in this paper. In this section, two fringe subtrees are regarded the same if there is a (root-preserving) isomorphism between the two.

\medskip

In the same way as for simply generated families of trees, we have to take the number of automorphisms into account, so there are now three factors that determine the probability that a random increasing tree in one of our very simple families is isomorphic to a fixed rooted unordered tree $t$:

\begin{itemize}
\item the number of plane representations of $t$, which is given by
$$\frac{\prod_v \deg(v)!}{|\Aut t|},$$
\item the weight
$$\prod_v \phi_{\deg(v)},$$
where $\phi_k = \binom{d}{k}$ for $d$-ary increasing trees, $\phi_k = \binom{r+k-1}{k}$ for gports, and $\phi_k = \frac{1}{k!}$ for recursive trees.
\item the number of increasing labellings of any plane representation, which is
$$\frac{|t|!}{\prod_v |t(v)|}.$$
\end{itemize}

The product of all these is proportional to the probability that a random increasing tree with $n = |t|$ vertices is isomorphic to $t$. One only needs to divide by the number (more precisely: total weight) of $n$-vertex increasing trees in the specific family to obtain the probability.

\medskip

So once again we consider a suitable additive functional that takes all these into account. For a tree $t$ whose root degree is $\rho(t)$ and whose branches belong to $k_t$ isomorphism classes with respective multiplicities $m_1$, $m_2$, \ldots, $m_{k_t}$, we define the toll function by
\begin{multline}\label{eq:noniso_toll}
f(t) = \log |t| + \log \big( m_1! m_2! \cdots m_{k_t}! \big) \\
- \begin{cases} \log d^{\underline{\rho(t)}} = \log \big(d(d-1)\cdots (d-\rho(t)+1)\big) & \text{$d$-ary increasing trees,} \\
\log r^{\overline{\rho(t)}} = \log \big(r(r+1)\cdots (r+\rho(t)-1)\big) & \text{gports,} \\
0 & \text{recursive trees.} \end{cases}
\end{multline}
Let $F$ be the associated additive functional. Then the probability that a random tree with $k$ vertices belongs to the same isomorphism class as a fixed $k$-vertex tree $t$ is
$$e^{-F(t)} \times \begin{cases} \frac{k!}{\prod_{j=1}^{k-1} (1+(d-1)j)} & \text{$d$-ary increasing trees,} \\ \frac{k!}{\prod_{j=1}^{k-1} ((r+1)j-1)} & \text{gports,} \\ k & \text{recursive trees.} \end{cases}$$
It is easy to see that the toll function $f(t)$ defined above is $O(\log |t| + \rho(t) \log \rho(t)) = O( \rho(t) \log |t|)$. So in order to show that the conditions of Lemma~\ref{lem:additive_functionals_increasing} are satisfied, one needs to bound the average root degree in a suitable way. For $d$-ary increasing trees, this is trivial. In the other two cases, one can use generating functions. 

\medskip

Recall that the exponential generating function $Y(x)$ for an increasing tree family satisfies the differential equation
$$Y'(x) = \Phi(Y(x)),$$
with $\Phi(t) = e^t$ for recursive trees and $\Phi(t) = (1-t)^{-r}$ for gports. The bivariate generating function $Y(x,u)$, in which $u$ marks the root degree, is given by
$$\frac{\partial}{\partial x} Y(x,u) = \Phi(u Y(x)),$$
thus
$$\frac{\partial^2}{\partial x \partial u} Y(x,u) \Big|_{u=1} = \Phi'(Y(x)) Y(x).$$
The average root degree of $k$-vertex trees is
$$\frac{[x^k] \frac{\partial}{\partial u} Y(x,u) \Big|_{u=1}}{[x^k] Y(x)} = \frac{[x^{k-1}] \frac{\partial^2}{\partial x \partial u} Y(x,u) \Big|_{u=1}}{k [x^k] Y(x)} = \frac{[x^{k-1}] \Phi'(Y(x)) Y(x)}{k [x^k] Y(x)} .$$
Plugging in $Y(x) = - \log(1-x)$ (for recursive trees) and $Y(x) = 1 - (1-(r+1)x)^{1/(r+1)}$ (for gports) respectively and simplifying, we find that the average root degree is $1 + \frac12 + \cdots + \frac1{k-1} \sim \log k$ for recursive trees and
$$\frac{(r+1)^{k-1} (k-1)!}{\prod_{j=2}^{k-1}((r+1)j+1)} - r \sim r \Gamma \Big( \frac{r}{r+1} \Big) k^{1/(r+1)}$$
for gports. Consequently, $\Ex|f(T_k)| = O(\log^2 k)$ and $\Ex|f(T_k)| = O(k^{1/(r+1)} \log k)$ respectively in Lemma~\ref{lem:additive_functionals_increasing}, which means that the conditions of that lemma are satisfied.

\medskip

We can conclude now as before that the conditions of Theorem~\ref{thm:master-theorem-increasing} hold. The number of non-isomorphic fringe subtrees is of the order $n/\log n$ for all families of increasing trees we are considering. For example, we obtain Theorem~\ref{thm:unorderedbst} as a corollary in the case of binary increasing trees, or equivalently, binary search trees (see the conference version of this paper \cite{SeelbachWagner20}).

For recursive trees, the upper bound of $O(n/\log n)$ was determined recently in a paper of Bodini, Genitrini, Gittenberger, Larcher and Naima \cite{BodiniGGLN20}. The authors of that paper conjectured that this upper bound is asymptotically sharp and proved a lower bound of order $\sqrt{n}$. Indeed, our general theorem (Theorem~\ref{thm:master-theorem-increasing}) applies and confirms their conjecture.

\begin{theo}\label{thm:unorderedrec}
Let $K_n$ be the total number of distinct unordered fringe subtrees in a random recursive tree of size $n$. For two constants $c_{17} \approx 0.9136401430$ and $c_{18} \approx 1.0837575972$, the following holds:
\begin{enumerate}
\item[(i)] $\displaystyle c_{17} \frac{n}{\log n} (1+o(1)) \leq \Ex(K_n) \leq c_{18} \frac{n}{\log n} (1+o(1))$,
\item[(ii)] $\displaystyle c_{17} \frac{n}{\log n} (1+o(1)) \leq K_n \leq c_{18} \frac{n}{\log n} (1+o(1))$ with high probability.
\end{enumerate}
\end{theo}

Here, the constant $c_{18}$ is the logarithm of the growth constant for the number of unordered rooted trees (P\'olya trees), see the proof of Theorem~\ref{cor:distinct_unordered_labelled} for comparison. The constant $c_{17}$ is more complicated: it is given by
$$c_{17} = \sum_{k =1}^{\infty} \frac{\Ex(f(T_k))}{k(k+1)},$$
where $T_k$ stands for a random recursive tree with $k$ vertices and $f$ is defined in~\eqref{eq:noniso_toll}. It seems difficult to determine the expected value $\Ex(f(T_k))$ exactly, and even numerical approximation is somewhat trickier than in the previous examples (however, it is easy to compute simple lower bounds, as it is clear that $\Ex(f(T_k)) \geq \log k$). Let us describe the approach:

\medskip

The component $\log |t|$ in \eqref{eq:noniso_toll} is easy to deal with and contributes $\sum_{k=1}^{\infty} \frac{\log k}{k(k+1)}$ to the constant $c_{17}$. In order to numerically compute the contribution of the rest, let us determine the probability that a specific rooted unordered tree $S$ occurs exactly $m$ times among the root branches of a $k$-vertex recursive tree. The contribution to $\Ex(f(T_k))$ will be precisely $\log m!$ times that probability. Let $s = |S|$ be the size of $S$, and let $p_S$ denote the probability that a random recursive tree of size $s$ is isomorphic to $S$. Then the bivariate exponential generating function $Y(x,u)$ for recursive trees where the second variable $u$ takes the number of isomorphic copies of $S$ as a root branch into account is given by
$$\frac{\partial}{\partial x} Y(x,u) = \exp \Big( Y(x,1) + \frac{(u-1)p_S}{s} x^s \Big).$$
Recall here that the coefficient of $x^s$ in $Y(x,1) = - \log(1-x)$ is $\frac{1}{s}$, so $\frac{p_S}{s} x^s$ represents the fraction that is isomorphic to $S$. For simplicity, set $c_S = \frac{p_S}{s}$. Then this reduces to
$$\frac{\partial}{\partial x} Y(x,u) = \frac{\exp((u-1)c_Sx^s)}{1-x}.$$
So the number of recursive trees of size $k$ in which precisely $m$ branches isomorphic to $S$ occur is
$$k! [x^k u^m] Y(x,u) = (k-1)! [x^{k-1} u^m] \frac{\partial}{\partial x} Y(x,u) = (k-1)! [x^{k-1}] \frac{c_S^m x^{ms}}{m!} \cdot \frac{\exp(-c_Sx^s)}{1-x}.$$
There are $(k-1)!$ recursive trees with $k$ vertices; thus we find that
$$\Ex(f(T_k)) = \log k + \sum_{m \geq 2} \sum_{S} \log m! [x^{k-1}] \frac{c_S^m x^{m|S|}}{m!} \frac{\exp(-c_Sx^{|S|})}{1-x}.$$
Since $\frac{1}{k(k+1)} = \int_0^1 \int_0^y x^{k-1} \,dx\,dy$, we get
$$c_{17} = \sum_{k=1}^{\infty} \frac{\Ex(f(T_k))}{k(k+1)} = \sum_{k=1}^{\infty} \frac{\log k}{k(k+1)} + \sum_{m \geq 2} \sum_{S} \log m! \int_0^1 \int_0^y \frac{c_S^m x^{m|S|}}{m!} \frac{\exp(-c_Sx^{|S|})}{1-x} \,dx\,dy.$$
Interchanging the order of integration, this becomes
$$c_{17} = \sum_{k=1}^{\infty} \frac{\log k}{k(k+1)} + \sum_{m \geq 2} \sum_{S} \log m! \int_0^1 \frac{c_S^m x^{m|S|}}{m!} \exp \big(-c_Sx^{|S|}\big) \,dx.$$
Lastly, expand the exponential function into a power series and integrate to obtain
$$c_{17} = \sum_{k=1}^{\infty} \frac{\log k}{k(k+1)} + \sum_{m \geq 2} \sum_{S} \frac{\log m!}{m!} \sum_{r=0}^{\infty} \frac{(-1)^r}{r!} \frac{c_S^{r+m}}{(r+m)|S|+1}$$
or equivalently
$$c_{17} = \sum_{k=1}^{\infty} \frac{\log k}{k(k+1)}  + \sum_S \sum_{\ell \geq 2} \Big( \sum_{m=2}^{\ell} (-1)^{\ell-m} \binom{\ell}{m} \log m! \Big) \frac{c_S^{\ell}}{\ell! (\ell|S| + 1)}.$$
The innermost sum actually simplifies to $\sum_{m=2}^{\ell} (-1)^{\ell-m} \binom{\ell-1}{m-1} \log m$ and only grows very slowly (it is $O(\log \log \ell)$, cf.~\cite[Theorem 4]{FlajoletS95}). Thus the sum over $\ell$ converges rapidly for every tree $S$. Moreover, it is $O(c_S^2)$ as $c_S \to 0$. One therefore gets a good numerical approximation by determining $c_S$ for small trees and only taking the sum over these small trees. For the ten digits given in the statement of the theorem, it was sufficient to consider trees $S$ with up to $20$ vertices.

\section{Conclusion}

Our main theorems are quite general and cover many different types of trees as well as different notions of distinctness. As the examples with explicit constants show, the upper and lower bounds they provide are typically quite close. Nevertheless, the following natural question arises from our results: for the random variables $J_n$ and $K_n$ as defined in Theorem~\ref{thm:plane-simply-generated} and Theorem~\ref{thm:simply-generated-unordered} respectively, are there always constants $c_J$ and $c_K$ such that
\begin{align*}
\Ex(J_n) =  \frac{c_J n}{\sqrt{\log n}}(1+o(1)), \qquad \Ex(K_n) =  \frac{c_K n}{\sqrt{\log n}}(1+o(1)),
\end{align*}
and
\begin{align*}
\frac{J_n}{n/\sqrt{\log n}} \overset{p}{\to} c_J, \qquad \frac{K_n}{n/\sqrt{\log n}} \overset{p}{\to} c_K\ ?
\end{align*}
In order to prove such estimates, it seems essential to gain a better understanding of the different additive functionals that we employed in the proofs of these theorems, in particular their distributions further away from the mean values. Analogous results for increasing trees would be equally interesting.

\bibliographystyle{plain}
\bibliography{bib}

\end{document}